\documentclass[12pt,a4paper]{amsart}

\usepackage{latexsym}

\usepackage[all]{xy} 
\usepackage{tikz,tikz-cd} 
\usetikzlibrary{decorations.pathreplacing} 
\usetikzlibrary{decorations.markings} 
\usetikzlibrary{arrows} 
\usetikzlibrary{decorations.pathmorphing} 

\usepackage{amssymb}
\usepackage{amsthm}
\usepackage{amsmath}
\usepackage{amstext}
\usepackage{amscd}
\usepackage{color}

 \usepackage{hyperref}

\newtheorem{thm}{Theorem}[section]
\newtheorem{lemma}[thm]{Lemma}
\newtheorem{prop}[thm]{Proposition}
\newtheorem{cor}[thm]{Corollary}
\newtheorem*{conjecture*}{Conjecture}

\theoremstyle{definition}
\newtheorem{defn}[thm]{Definition}

\newtheorem{example}[thm]{Example}

\newtheorem{rem}[thm]{Remark}
\newtheorem{question}[thm]{Question}

\usepackage[utf8x]{inputenc}

\DeclareMathOperator{\mult}{\rm {mult }}

\newcommand{\kk}{{k}}

\newcommand{\point}{P}
\newcommand{\dimtype}{\operatorname {d.t. }}

\usepackage{enumitem}
\setlist[itemize]{labelindent=.6em, itemindent=1em, leftmargin=!, label=\textbullet}

\newcommand{\R}{\mathbb R}
\renewcommand{\P}{\mathbb P}
\newcommand{\proj}{\mathbb P}
\newcommand{\N}{\mathbb N}

\newcommand{\C}{\mathbb{C}}
\newcommand{\K}{\mathbb{K}}
\newcommand{\Q}{\mathbb{Q}}
\newcommand{\D}{D}
\newcommand{\DD}{\Delta}
\newcommand{\pr}{pr }
\DeclareMathOperator{\unit}{unit }
\DeclareMathOperator{\Sing}{Sing \,}
\DeclareMathOperator{\Reg}{Reg }

\renewcommand{\a}{\alpha}
\renewcommand{\b}{\beta}
\newcommand{\q}{\mathbf{q}}
\newcommand{\x}{\mathbf{x}}

\newcommand{\la}{\lambda}
\newcommand{\ovl}{\overline}
\newcommand{\lgw}{\longrightarrow}

\newcommand{\W}{W }
\newcommand{\bW}{\overline {\W}}
\newcommand{\bbW}{\overline {\W^*}}

\renewcommand{\tt}{\mathbf{t}}

\DeclareMathOperator{\codim}{\rm {codim }} 

\newcommand{\bkk}{\overline {k}}
\newcommand{\bpoint}{\overline P}
\newcommand{\pointgen}{P_{gen}} 
\newcommand{\Ph}{\Phi }
\newcommand{\Ps}{\Psi }
\newcommand{\ph}{\phi }

\newcommand{\aaa}{arc-wise analytic}

\newcommand{\aab}{arc-wise analytically}

\title{Algebro-geometric equisingularity of Zariski}

\author{Adam Parusi\'nski}
\address{Universit\'e C\^ote d'Azur,  CNRS,  LJAD, UMR 7351, 06108 Nice, France} 
\email{adam.parusinski@univ-cotedazur.fr}
\thanks{
Partially supported  ANR project LISA (ANR-17-CE40-0023-03)}


\begin{document}

\begin{abstract}This is a survey on Zariski equisingularity. We recall its definition, main properties, and a variety of applications in Algebraic Geometry and Singularity Theory. In the first part of this survey, we consider Zariski equisingular families of complex analytic or algebraic hypersurfaces.  We also discuss how to construct Zariski equisingular deformations. 
 In the second part, we present Zariski equisingularity of hypersurfaces along a nonsingular subvariety and its relation to other equisingularity conditions. We also discuss the canonical stratification of such hypersurfaces given by the dimensionality type. \end{abstract}

 \maketitle

 \tableofcontents





\section*{Introduction}\label{sec:intro}

A singularity is the germ of a complex or real analytic space $(V,p )$ 
that is not regular at $p $.   Equisingularity means equivalent or similar singularity and it is always necessary to make precise which equivalence of singularities we have in mind.  Thus  two singularities $(X, x)$ and $(Y, y)$  are analytically equivalent if there is an analytic isomorphism germ 
$\ph: (X,x) \to (Y,y)$.  If $\ph$ is only a homeomorphism then we say 
that $(X, x)$ and $(Y, y)$ are topologically equivalent.  If 
$(X, x)$ and $(Y, y)$ are both subspaces of the affine space $(\K^n,0)$, 
$\K=\R$ or $\C$, then we may require $\ph $ to be the restriction of an isomorphism (resp. homeomorphism) of the ambient spaces $\Ph: (\K^n, 0) \to(\K^n, 0) $. If such $\Ph$ exists we say then that $(X, x)$ and $(Y, y)$ are 
ambient analytically (resp. topologically) equivalent.

Let $V$ be a real or complex analytic space.  Then there exists a stratification $\mathcal S$ of $V$,  that is a decomposition 
of $V$ into analytic manifolds that, moreover, are usually required to satisfy some additional properties. For the notion of stratification and a historical account of stratification theory, we refer the reader  to the paper of D. Trotman \cite{trotman2020} in the first volume of this handbook and the references therein.  It is known that there always exists a stratification of $V$ that is topologically equisingular  along each stratum, 
that is if $p_1$ and $p_2$ belong to the same stratum then 
$(V, p_1)$ and $(V, p_1)$ are topologically equivalent. 
If $V$ is a subspace of $\K^n$ then one may, moreover, require this 
stratification to be ambient topologically equisingular. This can be achieved by constructing a Whitney stratification of $V$.  
Another and entirely independent way of constructing such a stratification is Zariski equisingualrity, which is the subject of this survey.  


Recall that, in general, there is no stratification 
that is analytically equisingular along each stratum, as the classical example of Whitney 
\cite[Example 13.1]{whitneymorse65} shows: $V = \{(x,y,z)\in \K^3; xy(y+x) (y-zx)=0\}$ 
admits a continuous family of analytically (or even $C^1$-diffeomorphically) non-equivalent singularities, due to the phenomenon of continuous moduli 
(the cross-ratio in this example).  

In 1971 in "Some open questions in the theory of singularities" 
\cite{zariski71open}, O. Zariski proposed a general theory of equisingularity for complex algebraic and analytic hypersurfaces. 
Zariski's approach was based on a new version of equisingularity that Zariski called algebro-geometric equisingularity, since it was defined by purely algebraic means but it reflected many geometric properties.   For instance, as Varchenko shows in 
\cite{varchenkoizv72,varchenko73,varchenkoICM75} answering a question posed by Zariski, Zariski equisingularity, which we now call the algebro-geometric equisingularity of Zariski, implies ambient topological triviality.  Zariski equisingularity under an additional genericity of projection assumption implies Whitney conditions as shown by Speder 
\cite{speder75}. 


This notion of equisingularity extended Zariski's earlier work on the singularities of plane curves, their equivalence and their families, 
see \cite{zariski65-S1,zariski65-S2,zariski68-S3}. 
For the general case of hypersurface singularities over an algebraically closed field of characteristic zero Zariski presented his program in \cite{zariski76} and \cite{zariski1979}. The paper \cite{zariski1979} "Foundations of a general theory of equisingularity on $r$-dimensional algebroid and algebraic varieties, of embedding dimension $r + 1${"},  published in 1979, contains complete foundations of this theory, stated for algebroid varieties over an algebraically closed field of characteristic zero.  
(Recall that algebroid varieties are the varieties defined by ideals of the rings of formal power series, see \cite{lefschetz53} Ch. IV. 
 and \cite{zariski1979} Section 2.) Since then Zariski equisingularity has been widely applied in the theory of singularities.  We present in this survey an account, certainly incomplete, of this development. 

 Intuitively, Zariski's notion can be characterized by two properties:
\begin{enumerate}
	\item
If $(V,\point_1)$  and $(V,\point_2)$ are equisingular, then $\point_1$ is a regular point of $V$ if and only if $\point_2$ is a regular point of $V$.
\item
If $W\subset \Sing V$ is non-singular then $V$ is equisingular along $W$ at $\point\in W$ if and only if for all sufficiently general projections  
$\pi: \K^{r+1}\to \K^r$ the discriminant locus of $\pi|_V$ is equisingular 
along $\pi(W) $ at $\pi(\point)$.
\end{enumerate} 



Formally, one may talk about two notions of Zariski equisingularity, 
of a hypersurface along its nonsingular subvariety and of a family of hypersurface singularities parameterized by a finite number of parameters.  This is already present, implicitly, 
in \cite{zariski71open}, where the former one is motivated by the latter one. We shall follow this path in this survey as well. 
In Section \ref{sec:curves} we describe the equisingular families of complex plane curve singularities.  This description is based on Puiseux with parameter theorem, Theorem \ref{thm:PuiseuxTheorem}. Then we   introduce Zariski equisingularity of families, Sections \ref{sec:curves} and \ref{sec:infamilies}. As we have mentioned, Zariski equisingular families are topologically trivial.  Therefore Zariski equisingularity implies the generic topological equisingularity of real or complex, algebraic varieties or analytic spaces (not necessarily hypersurfaces).  We present this principle in subsection \ref{ssec:arbitrarycodimension}. As a consequence  Zariski equisingularity provides an algorithmic construction of a topologically equisingular stratification. More 
  applications to Algebraic Geometry are presented in subsections \ref{ssec:fundamentalgroup} and \ref{ssec:generalposition}.  

In Section \ref{sec:construction} we show how to construct equisingular deformations of a given singularity.  This construction appears in many applications, in particular it is used to show that a (real or complex) analytic singularity is homeomorphic to an algebraic one, see subsection \ref{ssec:homeoanalytic-algebraic}, and, moreover we may assume that the latter one is defined over the field of algebraic numbers $\overline \Q$, see subsection \ref{ssec:homeotonumberfields}.  
At the end of Section \ref{sec:construction} we discuss application of Zariski equisingularity to trivialize families of analytic function and map germs 
\ref{ssec:functions} and \ref{ssec:smoothmappings}.  

In Section \ref{sec:ZEalong}  we present the original notion of Zariski equisingularity of a hypersurface $V$ along a nonsingular subvariety $W$, and a related notion of the dimensionality type.  
Zariski equisingularity along a hypersurface is defined by taking successive co-rank 1 projections and their discriminants, and a similar construction is used to define Zariski equisingularity in families.  The main, and to some extent still open problem, is to decide what projection to take to verify  whether such an equisingularity holds.  As follows from Zariski work, in the case of  families of plane curves singularities, the equisingularity given by a single projection implies equisingularity for all transverse  projections, for this notion see subsection \ref{ssec:equimultiplicity}.  Therefore, originally, Zariski considered transverse projections as sufficient for such verification, see \cite{zariski71open}.  In 
\cite{luengo85} Luengo gave an example of a family of surface singularities in $\C^3$ that is Zariski equisingular for one transverse projection but not for a generic or generic 
linear projection.  Therefore, in \cite{zariski1979}, Zariski 
proposed to build this theory on the notion of "generic" projection.  The definition of such generic projection given in \cite{zariski1979} is 
therefore crucial.  It involves adding all the coefficients of a generic formal change of coordinates as indeterminates to the ground field.  As Zariski also showed in \cite{zariski1979} a generic (in a more standard meanining) polynomial projection gives the same theory, that is to say the same notion of generic Zariski equisingularity along a nonsingular subvariety.  But it is not known how to verify which polynomial projections are generic in this sense or even whether there is a bound on the 
minimal degree of such polynomial generic projections. This makes algorithmic computations of the dimensionality type and related notions of generic Zariski equisingularity and Zariski's canonical stratification impossible. The algebraic case was studied in more detail by Hironaka \cite{hironaka79}, where the algebraic semicontinuity of the minimal degree of such polynomial projection is shown.  
The question whether a generic linear projection is always sufficient is still open for dimensionality type $\ge 2$, though the case of the dimensionality type 2 is fairly well understood thanks to \cite{brianconhenry80}.  

\subsection*{General set-up}{}
In this survey we present Zariski's theory in the complex analytic
 set-up, which seems to be the most 
 common and of the biggest interest for singularity theory.  There are two obvious extensions that one has to keep in mind. The first one, as the original definition of Zariski, is the theory of algebroid varieties over an arbitrary algebraically closed field of characteristic zero, when one works with the varieties defined by the ideals in the ring of formal power series.  
The second one is the real analytic set-up.  
Many results on Zariski equisingularity, such as topological triviality for example, are valid in both complex analytic and real analytic set-ups.  The real analytic set up sometimes requires more careful statements, for instance, by replacing analytic sets by the equations or ideals defining them. In general, 
for Zariski equisingularity, the assumption on the ground field to be algebraically closed seems not to be essential, unlike the assumption to be of characteristic zero, which is necessary.  

In Section \ref{sec:curves}, which can be considered as a motivation for the general definition, we discuss the equisingularity of complex 
plane curves.  Section \ref{sec:infamilies}  and \ref{sec:construction} 
are presented for complex and real analytic or algebraic spaces.  
For the definitions, theorems and proofs of these two sections there is 
no essential difference between the real and the complex case.  
The second part of Section \ref{sec:ZEalong}, the dimensionality type,  
is presented in the algebroid set-up, like Zariski's original definition.  Every statement of this section holds in complex analytic case.  We also believe that it can be carried over to the real analytic set-up, but this has yet to be done.   

\subsection*{Notation and terminology.} 
We denote by $\K$ either  $\R$ or $\C$.  Thus, by $\K$-analytic we mean either real analytic or holomorphic (complex analytic).   Sometimes we abbreviate it saying that a space or a map is analytic if the ground field, 
$\C$ or $\R$, is clear from the context or if the result holds in both cases. 

By an analytic space, we mean one in the sense of \cite{narasimhan66}. 
As we work mostly in the local analytic case, it suffices to consider only analytic set germs.  For an analytic space $X$ by $\Sing X$ we denote the set of singular points of $X$, i.e. the support of the singular subspace of $X$.  By $\Reg X$ we denote its complement $X\setminus \Sing X$, the 
set of regular points of $X$.   For an analytic function germ $F$  we denote by  $V(F)$ its zero set and by $F_{red}$ its reduced (i.e. square free) form.  By a real analytic arc, we mean  a real analytic map $\gamma : I\to X$, where $I=(-1,1)$ and $X$ is a real or a complex analytic space. 

For a polynomial monic in $z$, 
$F(x,z) = z^d+ \sum_{i =1}^d a_i(x) z^{d-i}$, with coefficients analytic functions in $x$, we denote by $\D_F(x)$ its discriminant, and by 
$\DD_F(x)$ its discriminant locus, the zero set of $\D_F(x)$.  
The discriminant of $F$, and more generally the generalized discriminants of $F$, are recalled in Appendix, Section \ref{sec:discriminants}.  

We say that $f\in \K\{x\}$ is \emph{a unit} if $f(0)\ne 0$. 
 We often use Weierstrass Preparation Theorem.  Recall briefly its statement, see for instance 
\cite[Theorem 2, p. 12] {narasimhan66}, \cite[Ch. 3, \S 2]{lojasiewiczbook} for more details. Let $F(x,z) \in \K\{x,z\}$ be regular in the variable $z$, that is $F(0,z)= z^d \unit (z)$.  Then there are $a_i(x) \in \K\{x\}$ such that $a_i(0)=0$  and 
\begin{align*}
F(x,z) = \unit (x,z)\, (z^d+ \sum_{i =1}^d a_i(x) z^{d-i}).
\end{align*}
We call the monic polynomial $z^d+ \sum_{i =1}^d a_i(x) z^{d-i}$, \emph{the Weierstrass polynomial associated to $F$}.  An analogous statement holds for formal power series, i.e. for $F(x,z) \in \K[[x,z]]$.

\subsection*{Acknowledgements.} 
I would like to thank Jean-Baptiste Campesato, Clint McCrory, Laurentiu Paunescu, and Guillaume Rond for several remarks and suggestions concerning the earlier versions of this survey. I would like to thank  as well  the anonymous referee for many precise and helpful suggestions.


\section{Equisingular families of plane curve singularities}\label{sec:curves}

We recall the notion of equisingular families of complex plane curve singularities. 
There are several equivalent definitions that are proposed by Zariski in 
\cite{zariski65-S1,zariski65-S2}. 
 We use the one based on the discriminant of a local projection.  Firstly, this is the definition that Zariski generalizes to the higher-dimensional case.  Secondly, by Puiseux with parameter theorem, it gives an equiparameterization of such singularities by fractional power series.

Let 
\begin{align}\label{eq:weierstrass-F}
F(t,x,y) = y^d+ \sum_{i =1}^d a_i(t,x) y^{d-i} 
\end{align}
be a unitary polynomial in $y\in \C$ with complex analytic coefficients $a_i(t,x)$,  defined on $U_{\varepsilon,r} = U_{\varepsilon} 
\times U_r $, where $ U_\varepsilon = \{t\in \C^l;  \|t\|< \varepsilon\}$, $U_r= \{x\in \C ; | x|<r\}$ .  
Here $t=(t_1, \ldots, t_l)$ is considered as a parameter.  
One also often assumes that $F$ is reduced (has no multiple factors) 
so that its discriminant  $\D_{F}$ is not identically equal to zero. 
For arbitrary $F$ we either consider $\D_{F_{red}}$ or, equivalently, the first not identically equal to zero generalized discriminant of $F$, 
see Appendix, Section \ref{sec:discriminants}. 

\begin{thm}[Puiseux with parameter] \label{thm:PuiseuxTheorem}
 Suppose that the discriminant of $F_{red}$   
is of the form  $\D_{F_{red}} (t,x) = x^M \unit (t,x)$ where $\unit (t,x)$ is a complex analytic function defined and nowhere vanishing on $U_{\varepsilon,r}$.  Then there is a positive integer $N$ and complex analytic functions $\tilde \xi_i(t,u) $ defined on $U_{\varepsilon} 
\times U_{r^{1/N}} $ such that 
$$
F(t,u^N,y) = \prod _{i=1}^d(y- \tilde \xi_i (t,u)) .  
$$  
Let $\theta $ be an $N$th root of unity.  Then for each $i$ there is $j$ such 
that $\tilde \xi_i ( t,\theta u) = \tilde \xi_j (t, u)$.  
\end{thm}

If $F$ is irreducible then one can take $N=d$.  In general, $N=d!$ always works, but it is not minimal.  

 If $M=0$ then, by the Implicit Function Theorem (IFT), the roots of $F$, which we denote by $\xi_1 (t,x),  ... , \xi_d(t,x)$,  are $\C$-analytic functions of $(t,x)$. 
 Moreover two such $\xi_i$ and $\xi_j$ either coincide or are distinct everywhere. 
In general, for arbitrary $M$, Theorem \ref{thm:PuiseuxTheorem} implies that the projection of 
the zero set $V=V(F)$ of $F$ onto $U_{\varepsilon}$, given by $(x,y,t)\to t$  is topologically trivial.  
To see it one may use the following corollary.

\begin{cor}\label{cor:Puiseuxcor1}
For $x_0$ fixed, the  roots of $F$,  $\xi_1 (t,x_0),  \ldots, \xi_d(t,x_0)$, can be chosen complex analytic  in $t$. 
Moreover, if $\xi_i(0,x_0)=\xi_j(0,x_0)$   then $\xi_i(t,x_0)\equiv \xi_j(t,x_0)$.   Thus the multiplicity of each $\xi_i(t,x_0)$ as a root of $F$ is independent of $t$. 
\end{cor}

\begin{proof}
It suffices to show it for $F$ reduced.  Then for $x_0\ne 0$ it follows from the IFT.  Let us show it for $x_0=0$.  
The family $\xi_1 (t,0),  \ldots, \xi_d(t,0)$ coincides (as unordered sets) with $\tilde \xi_1 (t,0),  \ldots, \tilde \xi_d(t,0)$.  
If $\tilde \xi_i(0,0)= \tilde \xi_j(0,0)$ then $\tilde \xi_i(t,u)- \tilde \xi_j(t,u)$ is 
either identically zero or divides $u^{NM}$ and hence equals 
a power of $u$ times a unit.  
\end{proof}

Using Corollary \ref{cor:Puiseuxcor1} we may trivialize topologically $V$ with respect the parameter $t$ by 
\begin{align*}
\Ph (t, x, \xi_i(0,x)) = (t, x, \xi_i(t,x)) \quad i=1, \ldots, d.
\end{align*}
The map $\Ph$ is, by Corollary \ref{cor:Puiseuxcor1}, complex analytic in $t$ and  one can show, moreover, that it is a local homeomorphism. 
(It follows, for instance, from much more general \cite[Theorem1.2]{PP17}.)

The parameterized Puiseux Theorem, Theorem \ref{thm:PuiseuxTheorem}, can be proven in the same way as the classical  Puiseux Theorem by considering the finite covering of $V(F)$ over 
$U_{\varepsilon} \times  U_r^*$, where $U_r^*= U_r \setminus \{0\}$.  Then, for a positive integer $N$, the pullback of this covering by $(t,u) \to (t, u^N)=(t,x)$ is trivial, its sheets define the roots $\tilde \xi_i(t,u) $  that extend  analytically to 
$U_{\varepsilon} \times  \{0\}$ by Riemann's Removable Singularity Theorem.  For details 
we refer for instance to \cite{pawlucki84}. 
 Note also that this theorem is a special case of the Jung-Abhyankar Theorem, see \cite{jung1908}, \cite{abhyankar55}, that in complex analytic case can be proven exactly along the same lines, see \cite{PR2012} Proposition 2.1.

\begin{thm}[Jung-Abhyankar]\label{thm:Jung}
Let $\kk$ be an algebraically closed field of characteristic zero and let 
$f\in \kk[[x_1, \ldots ,x_{r+1}]]$ be of the form 
\begin{align*}
f(x_1, \ldots ,x_{r+1}) = x_{r+1}^d+ \sum_{i =1}^d a_i(x_1, \ldots ,x_{r}) x_{r+1}^{d-i}. 
\end{align*} 
Suppose the discriminant $\D_f$ of $f$ equals a monomial $\prod_{i=1}^k x_i ^{n_i}$ times a unit.  Then the roots of $f$  are fractional power series in $x_1, \ldots ,x_{r}$, 
More precisely, there is a positive integer $N$, such that the roots of 
$f(u_1^N, \ldots , u_k^N, x_{k+1}, \dots , x_{r+1})$ belong to 
$\kk[[u_1, \ldots , u_k, x_{k+1}, \dots , x_{r}]]$.    
\end{thm}

\subsection{Equisingular families of plane curve singularities. Definition.}
Let us fix a local projection $\pr: \C^{l}\times \C^2\to \C^l$ and suppose that $F$ is a complex analytic function defined 
in a neighborhood of the origin in $\C^{l}\times \C^2$ and vanishing 
identically on $T=\C^l \times \{0\}$.  We assume that $F$ is reduced and consider its zero set 
$V=V(F)=F^ {-1}(0)$ as a family of plane curve singularities 
\begin{align*}
t \rightsquigarrow (V_t,0)= (V\cap \pr^{-1} (t),0)
\end{align*}
 parameterized by $t\in (\C^l,0)$.  
 We say that a local system of coordinates 
$t_1, \ldots, t_l,x,y$ is \emph{$\pr$-compatible}  if $\pr (t,x,y) = t$, where $t=(t_1, \ldots, t_l)$, and $T=\{x=y=0\}$.  
Suppose that in such a system of coordinates $F$ is regular in variable $y$, 
i.e. $F(0,0,y) \not \equiv 0$. 
Then, by the Weierstrass Preparation Theorem, we may assume that, up to a multiplication by an analytic unit, $F$ is of the form \eqref{eq:weierstrass-F} with all $a_i(0,0) =0$. 
Note also that because $T\subset V(F)$, we have $F(t,0,0)\equiv 0$.

\begin{defn}\label{def:equising-curves}
We say that $V = V(F)$ is 
\emph{an equisingular family of plane curve singularities} if there are a $\pr$-compatible 
system of coordinates $t,x,y$, such that $F$ is regular in variable $y$, and a non-negative integer $M$, such that the discriminant  $\D_F (t,x)$ of $F$ 
is of the form 
\begin{align}\label{eq:discriminant}
\D_F (t,x) = x^M \unit(t,x). 
\end{align}  
\end{defn}

Equisingular families of plane curve singularities were studied in the 
algebroid set-up (i.e. defined by  $F$ being a formal power series) over an algebraically closed field of characteristic zero by Zariski  
\cite{zariski65-S1,zariski65-S2,zariski68-S3} mainly by means of 
(equi)resolution.  All the results of \cite{zariski65-S1,zariski65-S2,zariski68-S3}, properly stated, are valid for the complex analytic case. In particular, Zariski has shown that in such families the special fiber $(V_0,0)$ and the generic fiber $(V_{t_{gen}}, 0)$ are equivalent plane curve singularities, see 
\cite[section 6]{zariski65-S1} and \cite[section 3]{zariski65-S1} for several equivalent definitions of equivalent plane curve singularities. 
In the complex analytic set-up, two complex plane curve singularities are equivalent if and, only if they are ambient topologically equivalent. 
By \cite[Theorem 7]{zariski65-S1}, Zariski equisingular families of plane curve singularities are equimultiple, that is to say $\mult_{(t,0,0)} V$ is independent of $t$. If this multiplicity equals $d= \deg_y F$, then we say that 
the associated projection $\pi (x,y,t) = (x,t)$ is \emph{transverse}. Geometrically it means that the kernel of $\pi$ is not included in the tangent cone $C_0(V)$.  
Because the equimultiple families are normally pseudo-flat 
(continuity of the tangent cone),  it is enough to check the transversality 
for the special fiber $V_0$ and if it holds for the special fiber then it holds also for the generic one.  Zariski shows in Theorem 7 of \cite{zariski65-S1} also the following result.

\begin{thm}[\cite{zariski65-S1}, Theorem 7]
\label{thm:equisingulacurves}
If a family of plane curve singularities is equisingular (for a not necessarily transverse projection) then it is equisingular for all transverse projections. 
\end{thm}

Note that if $V=V(F)$ is equisingular then the singular locus $\Sing V$ of $V$ is 
$T=\C^l \times \{0\}$.  In \cite[section 8]{zariski65-S2} Zariski shows that a 
family $V(F)$ of plane curve singularities is equisingular if and only if  
$\Sing V = T$ and $V\setminus T , T$ is a Whitney stratification of $V$. In the complex analytic case, this gives another proof of the fact that the equisingular families of plane curve singularities are topologically trivial and the following holds.     


\begin{cor}\label{cor:Puiseuxcor2} 
The Puiseux pairs of the roots $\xi_i(t,x)$ and the contact exponents between different branches of $V_t$ are independent of $t$.  
\end{cor}

In the complex analytic set-up, in \cite{teissier77}[p. 623] B. Teissier gives 12 characterizations of equisingular families of plane curve singularities, including (equi)resolutions,  constancy of Milnor number, Whitney's conditions, and topological triviality.

\subsection{Equisingular families of plane curve singularities and Puiseux with parameter.} \label{ssec:puiseux-with-par}
The Puiseux with parameter theorem, Theorem \ref{thm:PuiseuxTheorem}, gives the following 
criterion of equisingularity of families of plane curve singularities.

\begin{thm}\label{thm:Jung2}
Let $F$ be reduced and of the form \eqref{eq:weierstrass-F} in a $\pr$-compatible system of coordinates $t,x,y$. We also assume $a_i(0,0)=0$ for all $i$.  Then $V(F)$ is an equisingular family of plane curve singularities for this system of coordinates if and only if there are $\tilde \xi_i \in \C\{t,u\}, i=1,\ldots ,d $, and strictly positive integers $N$, $k_{ij}, i<j$, such that 
\begin{align}
F(t,u^N,y) = \prod_{i=1}^d (y- \tilde \xi_i(t,u)) 
\end{align}
and $\tilde \xi_i - \tilde \xi_j = u^{k_{ij}} \unit (t,u) $ or $\xi_i$ and $\xi_j$ coincide everywhere  (the latter possibility may occur 
only if $F$ is not reduced). 
\end{thm}

The above observation implies, in particular, Corollary \ref{cor:Puiseuxcor1}.  
We also note that it implies that all $a_i $ of  \eqref{eq:weierstrass-F} satisfy $a_i(t,0)\equiv 0$.  
Indeed, by the assumption $T\subset V(F)$  there is a root $\tilde \xi_{j}  $ of $F$ such that 
$\tilde \xi_{j} (t,0) \equiv 0 $.  Since $\tilde \xi_{i} (0,0) =0  $ for all $i$ the 
last claim of the above theorem implies our assertion.


\section{Zariski equisingularity in families}\label{sec:infamilies}

Zariski equisingularity of families of singular varieties was introduced by 
Zariski in \cite{zariski71open} in the context of equisingularity of a hypersurface along a smooth subvariety, that we discuss in Section \ref{sec:ZEalong}.  This is a direct generalization of Definition 
\ref{def:equising-curves} but instead of a single co-rank one projection one considers a system of such subsequent projections.   It can be formulated over any field, in particular, in the analytic case over $\K=\R$ or $\C$.  
Recall that for $x=(x_1, \ldots, x_n) \in \K^n$ we denote $x^i = (x_1, \ldots, x_i) \in \K^i$.  

\begin{defn}\label{def:system}
 By a \emph{local system  of pseudopolynomials in $x=(x_1,... ,x_n)\in \K^n$ at $(0,0)\in \K^l\times \K^n$, with a parameter $t\in U \subset \K^l$}, we mean a family  of $\K$-analytic  functions 
  \begin{align}\label{def:pseudopolynomials}
  F_{i} (t, x^i )= x_i^{d_i}+ \sum_{j=1}^{d_i} a_{i-1,j} (t,x^{i-1})
 x_i^{d_i-j},  \quad i=0, \ldots,n, 
\end{align} 
defined on $U\times U_i$, where  $U_i $ 
is a neighborhood of the origin in $\K^i$, with the coefficients $a_{i,j}$ vanishing identically on $T=U\times \{0\}$.  
This includes $d_i=0$, in which case we mean $F_i\equiv 1$.    
\end{defn}

\begin{defn}\label{def:ZAinfamilies}
Let $V= F^{-1}(0)$ be an analytic hypersurface in a neighborhood 
of the origin in $\K^l \times  \K^n$.  
We say that $V$ is \emph{Zariski equisingular with respect to the 
parameter $t$ (and the system of coordinates $x_1, \ldots, x_n$)} 
 if there are $k\ge 0$ and a system of pseudopolynomials $F_{i} (t, x^i )$ such that 
\begin{enumerate}
\item 
$F_n$ is the Weierstrass polynomial associated to $F$.
\item
for every $i$, $k\le i\le n-1$, the discriminant of $(F_{i+1})_{red}$ (or, equivalently, the first not identically equal to zero generalized discriminant of $F_{i+1}$, see Appendix, Section \ref{sec:discriminants}) divides $F_{i}$.     
\item
$F_k \equiv 1$  (and then we put $F_i \equiv 1$ 
for all $0\le i< k$).
\end{enumerate}
\end{defn}

\begin{rem}
In the above definition, we suppose that the system of local coordinates 
$x_1, \ldots, x_n$ is fixed.  Of course one may say that $V$ is \emph{Zariski equisingular with respect to the parameter $t$}, if such a system exists.  This raises a variety of interesting 
 questions, for instance, how to check whether such a system exists.  We will discuss it in Section \ref{sec:ZEalong} in a slightly different set-up, Zariski singularity along a nonsingular subspace.  
\end{rem}

\begin{rem} 
In the original definition of Zariski equisingularity \cite{zariski71open} and also in \cite{varchenkoizv72}, \cite{varchenkoICM75}, the condition 2. was stated in an apparently  more restrictive way:
\begin{enumerate}
\item [2'.]
for every $i$, $k\le i\le n-1$,   $F_{i}$  is the Weierstrass polynomial associated to  the discriminant of $(F_{i+1})_{red}$.     
\end{enumerate}
The definition given here comes from \cite{PP17}   
and is often easier to work with than the original one. Probably, both definitions are equivalent. 
\end{rem}

\subsection{Topological equisingularity and topological triviality}\label{ssec:topequising}


In \cite{zariski71open} Zariski asked the following question: \medskip

\noindent
\emph{Does algebro-geometric equisingularity 
(i.e. Zariski equisingularity), in complex analytic case, imply topological equisingularity or even differential equisingularity?}
\medskip

 By the latter one, Zariski meant Whitney's conditions (a) and (b). The answer to this part of Zariski's question depends on how generic the system of coordinates giving Zariski equisingularity is, or equivalently how generic the projection defining the subsequent discriminants are, see section \ref{ssec:relationto} below.  In 1972 Varchenko \cite{varchenkoizv72} gave the affirmative answer to the first part of the question, see also  
\cite{varchenko73} and \cite{varchenkoICM75} for the statement of results.  

\begin{thm}\label{thm:topological_triviality} 
Suppose that $V$ is Zariski equisingular with respect to the parameter 
$t$.  Then there are 
neighborhoods $U$ of the origin in $\K^l$, $\Omega_0$ of the origin in $\K^n$, and $\Omega$ of the origin in $\K^{l+n}$, and  a 
homeomorphism 
\begin{align}\label{eq:homeomorphismPh}
\Ph  : U \times \Omega_0 \to \Omega, 
\end{align}
 such that 
 \begin{enumerate}
\item  [\rm (i)]
$\Ph (t, 0) =  (t,0) $,  $\Ph (0, x_1, \ldots, x_n) =  (0, x_1, \ldots, x_n) $;
\item [\rm (ii)]
$\Ph$ has a triangular form 
\begin{align}\label{eq:Phitriangular} 
\Ph (t, x_1, \ldots , x_n) = (t, \Ps_1(t, x_1), \ldots , \Ps_{n-1} (t,x_1, \ldots , x_{n-1}), \Ps_{n} (t, x_1, \ldots , x_{n}) );
\end{align}
\item [\rm (iii)]
$\Ph(U\times (V\cap \Omega_0))=V\cap \Omega$.  
\end{enumerate}
\end{thm}

We note that Varchenko's result gives local topological triviality, a property stronger than the topological equisingularity.  Here by topological equisingularity we mean the constancy of local topological types of $V_t: = V\cap (\{t\}\times \K^n)$ at the origin, i.e. the existence of homeomorphism germs $h_t : (V_0,0)\to (V_t,0)$, possibly given by ambient homeomorphisms $H_t : (\K^n,0) \to (\K^n,0)$. 
The (ambient) topological triviality, that is the existence of $\Ph$ of 
\eqref{eq:homeomorphismPh} implies that such $H_t(x) = \Ph (t,x)$ depends continuously on $t$.

The details of the proof of Theorem \ref{thm:topological_triviality}
 are published in  \cite{varchenkoizv72}.  Strictly speaking the proof in \cite{varchenkoizv72} is in the global polynomial case but it can be adapted easily to the local analytic case.  The homeomorphism $\Phi$ is constructed in the complex case $\K=\C$.   The real case follows from the complex one under a standard argument using the invariance by complex conjugation.    
The functions $\Ps_i$ are constructed inductively so that  every  
\begin{align}\label{eq:Phiinduction} 
\Ph_i (t,x_1, \ldots , x_i) = (t, \Ps_1(t, x_1), \ldots , \Ps_{i} (t,x_1, \ldots , x_{i}) )
\end{align}
induces topological triviality of $F_i^{-1}(0)$. Given $\Ph_i$, then 
$\Ph_{i+1}$ is constructed in two steps.\\

\noindent
\emph{Step 1.}  One lifts $\Ph_i$ to the zero set of $F_{i+1}^{-1}(0)$.  
Such a continuous lift exists and is unique thanks to the following lemma, cf. the multiplicity preservation lemmas of Section 2 of \cite{varchenkoizv72}  or Lemma on page 429 of \cite{varchenkoICM75}. This lemma and  the standard argument of the continuity of roots show that such a lift is continuous. 

\begin{lemma}\label{lem:multiplicityconstant}
Let \begin{align}\label{def:pseudopolynomial}
  F (t, x )= x_n^{d}+ \sum_{j=1}^{d} a_{i-1,j} (x^{n-1})
 x_n^{d-j},  
\end{align} be a pseudopolynomial defined in a neighborhood of $p=(p',p_n) \in \C^n$.  
Let $H_t : (\C^{n-1},p')\to 
(\C^{n-1}, p'_t)$, $t\in [0,1]$, be a continuous family of local homeomorphisms preserving the discriminant locus of $F$, that is 
$H_t (\DD_F, p') = (\DD_F, p'_t)$.  
Then the number of distinct roots of $F$ over $p'_t$, as well as their multiplicities, are independent of $t$.
\end{lemma}

\noindent
\emph{Step 2.} 
As soon as  $\Ph_{i+1}$ is defined on the zero set of $F_{i+1}^{-1}(0)$ it suffices to extend it to the ambient space.  This is obtained in Section 1 of \cite{varchenkoizv72} by the covering isotopy lemma, see also Fundamental Lemma of \cite{varchenko73}.  The construction of such extension is based on a triangulation of the base space, so that the finite branched covering $F_{i+1}^{-1}(0) \to \C^{l+i}$ is trivial over each open simplex, and a simplicial extension argument.


\subsection{Arcwise analytic triviality}\label{ssec:arcwise}
Zariski equisingularity implies much stronger triviality property than just the topological one. 
 The following result was shown in \cite[Theorem 3.1]{PP17}.

\begin{thm}\label{thm:arcwise-analytic_triviality}
Suppose that $V$ is Zariski equisingular with respect to the parameter $t$.  Then there are neighborhoods $U$ of the origin in $\K^l$, $\Omega_0$ of the origin in $\K^n$, and $\Omega$ of the origin in $\K^{l+n}$, and  a 
homeomorphism 
\begin{align}\label{eq:trivialization}
\Ph  : U \times \Omega_0 \to \Omega, 
\end{align}
 such that 
 \begin{enumerate}
\item  [\rm (i)]
$\Ph (t, 0) =  (t,0) $,  $\Ph (0, x_1, \ldots, x_n) =  (0, x_1, \ldots, x_n) $;
\item [\rm (ii)]
$\Ph$ has a triangular form \eqref{eq:Phitriangular} ;
\item [\rm (iii)]
there is $C>0$ such that  for all $(t,x)\in U \times \Omega_0$
\begin{align*}
  C^{-1} |F_n(\Ph (0, x))| \le |F_n(\Ph (t,x )) | \le  C |F_n(\Ph (0,x)) |   ;
\end{align*}
\item [\rm (iv)]
For $(t, x_1, \ldots , x_{i-1})$ fixed, 
$\Ps_i (t, x_1, \ldots , x_{i-1}, \cdot ): \K \to \K$ is bi-Lipschitz  
and the Lipschitz constants of 
$\Ps_i$ and $\Ps_i ^{-1}$ can be chosen independent of $(t, x_1, \ldots , x_{i-1})$;  
\item [\rm (v)]
$\Ph$ is an {\aaa} trivialization of the projection $\Omega \to U$ .  
\end{enumerate}
\end{thm}

(Note that (iii) of Theorem \ref{thm:arcwise-analytic_triviality} implies (iii)
Theorem \ref{thm:topological_triviality}.)  

Let us recall after \cite{PP17} the notion of arc-wise analytic trivialization.  First, we need to recall, after \cite{kurdyka88}, the notion of arc-analytic map.  Let $Y, Z$ be  real analytic spaces.  We say that a map 
$g(z) : Z\to Y$ is \emph{arc-analytic} if for everyreal analytic arc  $z (s) : I\to Z$,  $ g(z(s))$ is analytic in $s$.  Suppose now that 
$T, Y, Z$ are $\K$-analytic spaces, $T$ nonsingular.  We say that a map 
$f (t,z) : T \times Z\to Y$ is \emph{{\aaa } in} $t$ 
if it is $\K$-analytic in $t$ and arc-analytic in $z$, that is for every real analytic arc  $z (s) : I\to Z$,  the map 
$ f(t, z(s))$ is analytic in both $t$ and $s$.  Note that in the complex analytic case it means that $ f(t, z(s))$ can be written as a convergent power series $\sum _{\alpha=(\alpha_1, \ldots, \alpha_ l)} \sum _k  a_{\alpha,k}  t^ \alpha s^ k$ in $t$ complex and $s$ real.  

\begin{rem} 
In the complex analytic case, it is in general impossible to have 
the complex analytic dependence of $\Ph$ on $x$, even only on the complex arcs.  This rigidity property already appears for moduli spaces of elliptic curves.  
\end{rem}

Suppose now that, moreover, $\pi : Y\to T$ is $\K$-analytic.  
 We say  $$\Ph (t,z) : T \times Z\to Y$$
  is an  \emph{{\aaa} trivialization of $\pi$}, see  \cite[Definition 1.2]{PP17}, if it satisfies the following properties 
  \begin{enumerate}
  \item 
  $\Ph$  is a subanalytic homeomorphism (semi-algebraic in the algebraic case),   
   \item 
  $\Ph$ is {\aaa} in $t$ (in particular it is $\K$-analytic with respect to $t$), 
  \item
   $\pi \circ \Ph (t,z)=t$  for every  $(t,z) \in T\times Z$, 
     \item
   the inverse of $\Ph$ is arc-analytic, 
    \item 
  there exist $\K$-analytic stratifications $\{Z_i\}$ of $Z$ and $\{Y_i\}$ of $Y,$ such that for each $i$,   
  $Y_i = \Ph (T\times Z_i)$  and $\Ph_{|T\times Z_i} : T\times Z_i \to Y_i$ 
  is a real analytic diffeomorphism.  
   \end{enumerate}

The proof of Theorem \ref{thm:arcwise-analytic_triviality} follows Varchenko's strategy, \cite{varchenkoizv72}, which we recalled briefly in subsection \ref{ssec:topequising}.  It is technically simpler since in  Step 2 of the proof, the extension of the trivialization to the ambient space, is based on Whitney Interpolation Formula, 
see \cite{whitneymorse65}, \cite[Appendix I]{PP17}.  The homeomorphism 
$\Ph_{i+1}$ is given by a precise algebraic formula (formula (3.5) of 
\cite{PP17}
) in terms of the roots of the pseudopolynomial  $F_{i+1}$ and $\Ph_{i}$.   (This algebraic formula is a real rational map, it involves, in particular, the square of the distance to the roots of $F_{i+1}$.   
There is no such a complex rational formula and no hope, of course, to make $\Ph_{i+1}$ complex arc-analytic, that would have been, by Hartog's Theorem. just complex analytic.)

The fact that thus obtained trivialization $\Ph_{i+1}$ is {\aaa } is 
  proven by induction on $i$.  The inductive step is obtained by
   a reduction to the Puiseux with parameter theorem, Theorem \ref{thm:PuiseuxTheorem}. Let $x^{i}(s)$ be a real analytic arc.  By the inductive assumption $\Ph_{i} (t,x^i(s))$ is analytic in $t,s$.  Therefore $P(t,s,x_{i+1})=F_{i+1}(\Ph_{i} (t,x^{i}(s), x_{i+1}))$  is a pseudopolynomial with respect ot $x_{i+1}$ depending analytically on   $s$ and $t$.  The main point of the proof is to show that $P(t,s,x_{i+1})$ defines a Zariski equisingular family of plane curve singularities parameterized by $t$.  It follows in essence by the 
  stability of the discriminant by a base change, though technically it is more involved, $P$ is not necessarily reduced even if so is $F_{i+1}$, see the proof of 
\cite[Theorem 3.1]{PP17} 
for more details.   

\begin{rem}
Arc-wise analytic triviality is, in part, motivated by the relation of Zariski equisingularity and equiresolution of singularities and the theory of blow-analytic equivalence, see the last paragraphs of subsection \ref{ssec:equiresolution}.  
 \end{rem}


\subsection{Whitney's Fibering Conjecture.}\label{ssec:fiberingconjecture}
In \cite{PP17}, Theorem \ref{thm:arcwise-analytic_triviality} is used to 
show Whitney's fibering conjecture.

Whitney stated this conjecture in the context of 
the regularity conditions (a) and (b) introduced in  
\cite{whitneyannals65}.  These conditions on stratification imply the topological triviality along each stratum.  This trivialization is obtained by the flow  of "controlled" vector fields as follows from proofs of Thom-Mather Isotopy Lemmas.  By stating the Fibering Conjecture, Whitney wanted a stronger version of triviality, namely that the stratified set locally fibers into submanifolds isomorphic to strata.   

\begin{conjecture*}[Whitney's fibering conjecture, \cite{whitneymorse65} section 9, p.230] \label{conjecture}

Any analytic subvariety $V\subset U$ ($U$ open in $\C^n$) 
has a stratification such that each point $p_0\in V$ has a neighborhood 
$U_0$ with a semi-analytic fibration. 
\end{conjecture*}

By a semi-analytic fibration Whitney meant a local trivialization as in 
\eqref{eq:trivialization} that depends analytically on the parameter $t$.
Whitney does not specify the dependence on $x$, besides that he requires it to be continuous and that the existence of such fibration should imply 
Whitney's regularity conditions (a) and (b) (Whitney's semi-analytic fibration should not be confused with the notion of semi-analytic set introduced about the same time by 
{\L}ojasiewicz in \cite{lojasiewiczIHES}). Partial results on Whitney's fibering conjectures were obtained in \cite{hardtsullivan88}, and in the smooth case in \cite{muroloduplesseistrotman18}.

Whitney's fibering conjecture was proven in \cite{PP17} 
in the local complex and real analytic cases and in global algebraic cases  by means of Zariski equisingularity and arc-wise analytic triviality.  
More precisely, by Theorem \ref{thm:arcwise-analytic_triviality},  
 every such set has a stratification that 
locally admits arc-wise analytic trivializations (see the previous subsection) along each stratum. Existence of such trivialization guarantees Whitney's regularity condition (a) but not necessarily condition (b)  
(see Brian\c con-Speder example, Example \ref{ex:briancon-speder} below). 
Here we touch for the first time in this survey an interesting and important feature, some properties of Zariski equisingular families depend on the genericity of the system of coordinates $x_1, \ldots ,x_n$.  In order to guarantee Whitney's condition (b) we consider transverse Zariski equisingularity, see \cite[Definition 4.1]{PP17}.  We call Zariski equisingularity \emph{transverse (or transversal)} if at each inductive stage the kernel of the projection $(t, x^i) \to (t, x^{ i-1})$ is not included in the tangent cone to $F_i = 0$ at the origin. 
If we have a family that is transverse Zariski equisingular then,
 by Theorem 4.3 of \cite{PP17}, the arc-wise analytic trivialization constructed in \cite{PP17} 
satisfies additionally the property, called regularity, 
  \begin{align*}
  C^{-1} \|x\| \le \|\Ph (t,x) - (t,0)\| \le  C \|x\| ,
\end{align*}
for a constant $C$ independent of $t$ and $x$. 
 Geometrically it means that the trivialization $\Ph$ preserves the magnitude of the distance to 
$U\times \{0\}$.  It is proven in 
\cite[Proposition 7.4]{PP17} 
that this regularity implies Whitney's regularity conditions 
\cite[Section 7]{PP17}
) . 
condtion (b) (and even Verdier's condition (w)) along $U\times \{0\}$.

We discuss the relation of Zariski equisingularity, the plain one or with extra conditions such as transversality and genericity, and 
Whitney's conditions in subsection \ref{ssec:relationto}. 

\subsection{Algebraic Case}\label{ssec:algebraic}

In the papers \cite{varchenko73,varchenkoICM75} Varchenko considers the families of analytic singularities while the paper \cite{varchenkoizv72} deals with the families of affine or projective algebraic varieties.  Similarly, the  families of algebraic varieties were  considered in sections 5 and 9 of \cite{PP17}.  The version presented below is stated in \cite{PRpreprint18}.  It follows 
from the proof of the main theorem, Theorem 3.3, of \cite{PP17}, see also  Theorems 3.1  and 4.1 of \cite{varchenkoizv72} in the complex case, Theorems 6.1 and 6.3 \cite{varchenkoizv72} in the real case, and Proposition 5.2 and Theorem 9.2 of \cite{PP17} where the global algebraic case is treated.


\begin{thm}\label{thm:algebraic}
Let $\mathcal V$ be an open connected neighborhood of $\tt$ in $\K^r$ and 
let $\mathcal O_{\mathcal V}$ denote the ring of $\K$-analytic functions on 
$\mathcal V$.  
Let $t=(t_1, \ldots , t_r)$ denote the variables in $\mathcal V$ and 
let $x=(x_1,\ldots, x_n)$ be a set of variables in $\K^n$. Suppose that for $i=k_0,\ldots, n,$ there are given 
\begin{align}\label{eq:system-extension}
F_i(t,x^i)=x_i^{d_i}+\sum_{j=1}^{d_i}a_{i-1,j}(t,x^{i-1})x_i^{d_i-j}\in 
\mathcal O_{\mathcal V}[x^{i}], 
\end{align}
with $\ d_i>0$,    such that

\begin{enumerate}
\item[(i)] for every $i>k_0$, the first non identically equal to zero generalized discriminant of $F_i(t,x^{i-1},x_i)$ 
divides  $F_{i-1} (t,x^{i-1})$.

\item[(ii)] the first non identically equal to  zero generalized discriminant of $F_{k_0}$ is independent of $x$ and does not vanish on $\mathcal V$.
\end{enumerate}
Then,  for every $\q \in\mathcal V$ there is a homeomorphism
$$h_{\q}: \{\tt\}\times\K^n\rightarrow \{\q\}\times\K^n$$ such that
$h_{\q} (V_{\tt})=V_{\q}$, 
where for ${\q} \in \mathcal V$ we denote $V_\q=\{(\q,\x)\in \mathcal V \times \K^n\mid F_n(\q,\x)=0\}$.  

Moreover, if $F_n=G_1\cdots G_s$ then for every $j=1,\ldots, s$  
$$h_{\q}\left(G_j^{-1}(0)\cap(\{\tt\}\times\K^n)\right)=G_j^{-1}(0)
\cap(\{\q \}\times\K^n).$$
\end{thm}

\begin{rem}
By construction of \cite{varchenkoizv72, varchenkoizv73,varchenkoICM75} and \cite{PP17}, the homeomorphisms $h_\q$ can be obtained by a local topological trivialization. That is there are a neighborhood $\mathcal W$  of $\tt$ in $\K^r$ and a homeomorphism 
$$
\Ph : \mathcal W \times \K^n \to \mathcal W \times \K^n,  
$$
so that $\Ph (\q,\tt,x)) = h_{\q}(\tt,x)$. 
This $\Ph $ is triangular of the form \eqref{eq:Phitriangular}.   If we 
write $\Ph(\q,x)= (\q, \Ps _\q (x)) $, i.e. as a family of homeomorphisms 
$h_\q= \Ps _\q :\K^n \to \K^n$, then, as follows from \cite{PP17}, 
we may require that:  
\begin{enumerate}
	\item
	The homeomorphism $\Phi$ is subanalytic.  In the algebraic case, 
	i.e. if  we replace in the assumptions  
	$\mathcal O_{\mathcal V}$ by the ring of regular or 
	$\K$-valued Nash functions on 
$\mathcal V$, $\Ph $ can be chosen semialgebraic. 
\item
$\Ph$ is arc-wise analytic. In particular, each $h_\q$  and its inverse 
$h_\q^{-1}$ are arc-analytic.  
\end{enumerate}
\end{rem}

\begin{rem}\label{rk:proj}
If $F_i$ are homogeneous in $x$, then the functions $\Ps_\q$ satisfy, by construction, 
$$\forall \la\in\K^*,\forall x\in \K^n\ \ \ \Psi_\q(\la x)=\la \Psi_\q(x).
$$
Hence if we define $\P (V_\q)=\{(\q,\x)\in \mathcal V \times \P_\K^n\mid F_n(\q,\x)=0\}$, the homeomorphism $h_\q$ induces a homeomorphism between $\P(V_{\tt})$ and $\P(V_{\q})$.
\end{rem}

\subsection{Principle of generic topological equisingularity}\label{ssec:arbitrarycodimension}

Varchenko applies Theorem \ref{thm:topological_triviality} to establish in \cite[Sections 5 and 6]{varchenkoizv72} generic topological equisingularity for families of 
real or complex, affine or projective, algebraic sets.  
The principle of generic topological equisingularity says that 
in an algebraic family $X_t$ of algebraic sets, parameterized by $t\in T$, where $T$ is not necessarily nonsingular, irreducible algebraic variety, 
there is a proper algebraic subset $Y$ of $T$ such that 
 the fibers $X_t$ have constant topological type for $t$ from each  connected component of $T\setminus Y$.   
 In the complex algebraic case $T\setminus Y$ is connected by the irreducibility of $T$, in the real algebraic case it has finitely many connected components.  For analytic spaces or sets, a similar principle holds locally. In both analytic and algebraic cases, the results give actually local topological triviality of the family $X_t$ over $T\setminus Y$.  We give examples of possible precise statements below.  Let us first make some remarks. 

Generic topological equisingularity can be proven, in general, either by Zariski equisingularity or by stratification theory using  Whitney's stratification and Thom-Mather Isotopy Lemmas, see \cite[Theorem 4.2.17]{trotman2020}. Whitney's stratification approach is independent of the choice of coordinates and simple to define.  
But the trivializations obtained by this method are not explicit since they are flows of "controlled" vector fields.  Even if such vector fields can be chosen subanalytic or semialgebraic,  not much can be said about the regularity of their flows. 
Zariski's  equisingularity method is more explicit and in a way constructive.  It uses the actual equations and coordinate systems.  This can be considered either as a drawback or as an advantage.   
The trivializations can be chosen subanalytic (semialgebraic in the 
algebraic case), as shown in \cite{PP17}.  Actually, the trivializations 
are given there by explicit formulas in terms of the coefficients of the polynomials and their roots. 

In the real case, the triangulation provides another method for proving generic topological triviality.  The classical triangulation procedures are based on a similar construction as Zariski equisingularity, i.e. subsequent co-rank 1 projections and their discriminants. For instance, a beautiful result on semialgebraic triviality was shown by Hardt using this approach in \cite{hardt80}. For a fairly complete account on this approach, the reader can consult \cite{BCR} and the references therein. 

It is fairly straightforward to apply Theorem \ref{thm:topological_triviality}  to obtain generic topological equisingularity  for families of hypersurfaces.  
In the case of varieties and spaces of arbitrary codimension the argument 
goes as follows.  
If $F= G_1 \cdots G_k$, then, under the assumptions of Lemma \ref{lem:multiplicityconstant},  for $x^{n-1}$ fixed, the 
 number of roots of each  $G_{j} (\Ph_{n-1} (t,x^{n-1}), x_n)=0$ is independent of $t$, see Lemma 2.2 of  \cite{varchenkoizv72} or  Proposition 3.6 of \cite{PP17}.  In particular $\Ph $ 
 trivializes not only $V(F)=F ^{-1}(0)$ but also each of 
 $V(G_j)= G_j ^{-1}(0)$.   
Thus  \cite{varchenkoizv72}  implies the following.

 \begin{thm}\label{thm:many-components} 
  If $F= G_1 \cdots G_k$ then for each $j=1, \ldots , k$, 
  the homeomorphisms $\Ph$ of Theorem \ref{thm:topological_triviality} satisfies $\Ph(T\times V(G_j )\cap \Omega_0))=V(G_j) \cap \Omega$, 
 where 
$V_t(G_j) = (G_j^ {-1} (0) \cap  (\{t\}\times \K^n)$. In particular $\Ph$ trivializes $\{G_1= \cdots = G_k=0\}$.  
\end{thm}

Now let us give two possible exact statements for this principle taken  from \cite{PP17}.  
Note that they give not only generic topological equisingularity but 
much stronger generic arc-wise analytic triviality.  

\begin{thm}[{\cite[Theorem 9.3]{PP17}}, cf. {\cite[Theorem 5.2, Theorem 6.4]{varchenkoizv72}}]\label{thm:theoremequisingularity3}
Let $T$ be an algebraic variety (over $\K$)  and let $\mathcal X =\{X_k\}$ be  a finite family of 
algebraic subsets $T \times \proj^{n-1}_\K$.  Then there exists 
an algebraic 
stratification  $\mathcal S$ of $T$ such that for every stratum $S$ and for 
every $t_0\in S$  there is a neighborhood $U$ of $t_0$ in $S$ and a semialgebraic {\aaa } 
 trivialization  of $\pi,$ preserving each set of the family $\mathcal X,$
\begin{align}\label{eq:raagtrivial}
\Ph  : U \times \proj^{n-1}_\K \to \pi^{-1} (U),\end{align}
$\Ph(t,x)= (t,\Ps(t,x))$, $\Ph(t_0,x)= (t_0,x)$,  
where $\pi:T \times \proj^{n-1}_\K\to T$  denotes the projection.   
\end{thm}

\begin{thm}[{\cite[Theorem 6.2]{PP17}}]\label{thm:genericequisingularity}
Let $T$ be a $\K$-analytic space, $U\subset \K^n$ an open neighborhood of the origin, 
$\pi : T\times U \to T$ the standard projection,  and let 
$\mathcal X=\{X_k\}$ be a finite family  of $\K$-analytic subsets of $T\times U$. 
Let $t_0\in T$.   
Then there exist  an open  neighborhood $T' $ of $t_0$ in $T$ and a proper $\K$-analytic subset $Z\subset T'$, containing $\Sing T'$,  such that for every $t\in T'\setminus Z$,  $\mathcal X$ 
is  regularly {\aab }  equisingular along $T\times \{0\}$ at $t$.  

Moreover, there is an analytic stratification of an open neighborhood of $t_0$ in $T$ such that for every stratum $S$ and every $t\in S$,  $\mathcal X$ 
is  regularly {\aaa }  equisingular along $S\times \{0\}$ at $t$.
\end{thm}

In the above theorem by saying that $\mathcal X$ 
is \emph{{\aab }  equisingular along $T\times \{0\}$ at $t\in \Reg T$} we mean that there are neighborhoods  $B $ 
of $t$ in $\Reg T$ and $\Omega$ of  $(t,0)$ in $T\times \K^n$, and an {\aaa } trivialization 
$\Ph : B \times \Omega_t \to \Omega$, where $\Omega_t = \Omega\cap \pi^{-1} (t)$, such that 
$\Ph (B\times \{0\} ) =  B\times \{0\} $ and for every $k$, $\Ph (T\times X_{k,t}) = X_k, $ 
where $X_{k,t}= X_k \cap \pi^{-1} (t)$.  We say that  $\mathcal X$ 
is  \emph{regularly {\aab }  equisingular along 
$T\times \{0\}$ at $t\in T$} if, moreover, $\Ph$ preserves, up to a constant, the distance to 
$T\times \{0\}$, as we explained at the end of section 
\ref{ssec:fiberingconjecture}. The latter property is related to Whitney's conditions, 
see Section 7 of \cite{PP17}.  


\subsection{Zariski's theorem on the fundamental group.} \label{ssec:fundamentalgroup}
Varchenko in \cite{varchenkoizv72} applies topological triviality of Zariski equisingular projective algebraic varieties to prove Zariski's theorem on the fundamental group  of the complement.  This theorem says that the fundamental group of the complement $\proj^n_\C \setminus V_{n-1}$ of a complex projective hypersurface $V_{n-1}$, $n>2$, coincides with the corresponding group obtained from a general hyperplane section. 

This theorem was announced by Zariski in \cite{zariski37},
but the proof published in it is not considered as complete.  
Another complete proof of this theorem, different from the one of  Varchenko, is given in \cite{hammle71,hammle73}. 


 \subsection{General position theorem.} \label{ssec:generalposition}
In \cite{mccroryetal19} Zariski equisingular families of affine or projective algebraic varieties are used, together with Whitney interpolation, to prove stratified general position and transversality theorems for semialgebraic subsets of algebraic stratifications.

In classical algebraic topology, general position of chains was used by Lefschetz to define the intersection pairing on the homology of a manifold.
This approach is based on a possibility of moving a "subvariety" $Z$ of a 
$C^\infty$ manifold $M$, by a family of diffeomorphisms, so that the image $Z$ becomes transverse to a given another "subvariety" $W$ of $M$.  This principle was made precise by Trotman 
\cite{trotman79} 
and, independently, by Goresky 
\cite{goresky81}.  They proved that by a diffeomorphism one can put in  a stratified general position two Whitney stratified closed subsets $Z$ and $W$ of $M$.  

The main theorem of \cite{mccroryetal19} is  expressed 
in terms of a submersive family of diffeomorphisms introduced in 
\cite[I.1.3.5]{StratifiedMorseTheory}. 
Let $T$ and $M$ be $C^\infty$ manifolds  and let $\Ps : T\times M  \to M$ be a $C^\infty$ map. 
Consider  $\Ps_t : M \to M$, $\Ps_t(x) = \Ps (t, x)$, and $\Ps ^x : T \to  M$, $\Ps ^x(t) = \Ps (t, x)$.   We say  $\Ps$ is \emph{a family of diffeomorphisms} if for all $t\in T$ the map $\Ps_t$ is a diffeomorphism. The family $\Ps $ is called \emph{submersive} if, for each $(t, x) \in T\times M$, the differential $D\Ps ^x$ is surjective. 
By Theorem \mbox{
\cite[I.1.3.6]{StratifiedMorseTheory}}
, if $\Ps : T\times M  \to M$ is submersive and both $Z$ and $W$ are Whitney stratified closed subsets 
of $M$ then the set of $t\in T$ such that $\Ps_t(Z)$ is transverse to $W$ is dense in $T$ and open provided $Z$ is compact.  
 A good example of a submersive family is a transitive action  $\Ps : G\times M  \to M$ of a Lie group.  Note that in this case Theorem 
\cite[I.1.3.6]{StratifiedMorseTheory} 
gives characteristic $0$ part of Kleiman's transversality of a general translate theorem 
\cite{kleiman74}. For a stratified set $X= \bigsqcup S_i$ we say that $\Ps  : T \times X \to X$ is \emph{a stratified submersive family of diffeomorphisms} 
if for each stratum $S_j$,  we have $\Ps(T\times S_j)\subset S_j$, and the map $\Ps : T \times S_j  \to S_j$ is a submersive family of diffeomorphisms. 

 In algebraic geometry the intersection of cycles can be defined via a moving lemma that allows to move the cycle of nonsingular varieties, 
 see \cite[Section 11.4]{F:IT}. 
 But there is no moving lemma nor algebraic general position theorem for singular varieties. In the original construction of Intersection Cohomology \cite{goreskymacphersonIH} 
in order to define the intersection pairing on singular complex algebraic varieties equipped with a Whitney stratification Goresky and MacPherson used a piecewise linear general position theorem of McCrory 
\cite{mccrory78}.   The main theorem of \cite{mccroryetal19} 
shows the existence of such stratified submersive family in the arc-wise analytic category of \cite{PP17}, see also subsection \ref{ssec:arcwise}.

\begin{thm}[{\cite[Theorem 1.1]{mccroryetal19}}]\label{thm:projective}
Let $\mathcal V = \{V_i\}$ be a finite family of algebraic subsets of projective space $\proj_\K^n$. There exists an algebraic stratification $\mathcal S = \{S_j\}$ of $\proj_\K^n$ compatible with each $V_i$ and a semialgebraic stratified submersive family of diffeomorphisms  $\Ps : U\times \proj_\K^n  \to \proj_\K^n$,  where $U$ is an open neighborhood of the origin in $\K^{n+1}$, 
such that $\Ps(0,x)=x$ for all $x\in \proj_\K^n$.  Moreover, the map $\Ph:U\times \proj_\K^ n  \to U\times \proj_\K^n$, $\Ph(t,x) = (t,\Ps(t,x))$, is an arc-wise analytic trivialization of the projection $U\times \proj_\K^n\to U$.
\end{thm}

A similar result holds for affine varieties; see \cite[Corollary 3.2]{mccroryetal19}.

The proof of Theorem \ref{thm:projective} is rather tricky.  
It uses the formulas used in \mbox{
\cite{PP17} }
in  Step 2  of the construction topological trivialization of Zariski equisingular families, see  section \ref{ssec:topequising}.  These formulas are based on Whitney interpolation and can be perturbed by introducing complex parameters, 
these are $t\in U$ of the theorem.  The whole construction is applied to a  trivial family, that is to the product $U\times \proj_\K^n$, thus producing a non-trivial arc-wise analytic trivialization of a trivial family.

Theorem \ref{thm:projective} implies the general position in terms of the expected dimension of the intersection and the general transversality. 
The general position in terms of dimension  is exactly what is needed to define the intersection pairing for the intersection homology, 
cf.  \cite{goreskymacphersonIH}. 
The general position in terms of dimension can be expressed as follows, the dimension means the real dimension since we consider semialgebraic sets.

\begin{cor} [{\cite[Proposition 1.3]{mccroryetal19}}]\label{cor:generalposition1}
Let $\Ps : U\times \proj_\K^ n  \to \proj^n$ be a stratified family as in Theorem \ref{thm:projective}, and let $\mathcal S$ be the associated algebraic stratification of $\proj_\K^n$.  
 Let $Z$ and $W$ be semialgebraic subsets of $\proj_\K^n$. There is an open dense semialgebraic subset $U'$ of $U$ such that, for all $t \in U'$ and all strata $S\in \mathcal S$,
 $$
\dim   (Z \cap \Ps_t^{-1} (W )\cap S) \le  \dim (Z\cap S)  + \dim (W\cap S)  - \dim S. 
$$
\end{cor}

If $\mathcal S$ is a stratification of a semialgebraic set $X$, and $\mathcal T$ is a stratification of a semialgebraic subset $Y$ of $X$, 
then $(Y,\mathcal T)$ is a \emph{substratified object} of $(X,\mathcal S)$ if each stratum of $\mathcal T$ is contained in a  stratum of $\mathcal S$. Two substratified objects $(Z,\mathcal A)$ and $(W,\mathcal B)$ of $(X,\mathcal S)$ are   \emph{transverse} in $(X,\mathcal S)$ if, for every pair of strata $A\in\mathcal A$ and $B\in\mathcal B$ such that $A$ and $B$ are contained in the same stratum $S\in\mathcal S$, the manifolds $A$ and $B$ are transverse in $S$.

\begin{cor}[{\cite[Proposition 1.5]{mccroryetal19}}]\label{cor:transversality}
Let $\Ps : U\times \proj^ n  \to \proj^n$ be a stratified family as in Theorem \ref{thm:projective}, and let $\mathcal S$ be the associated algebraic stratification of $\proj^n$.  
 Let $Z$ and $W$ be semialgebraic subsets of $\proj^n$, with  semialgebraic stratifications $\mathcal A$ of $Z$ and $\mathcal B$ of $W$ such that $(Z,\mathcal A)$ and $(W,\mathcal B)$ are substratified objects of $(\proj^n,\mathcal S)$. There is an open dense semialgebraic subset $U'$ of $U$ such that, for all $t \in U'$, $(Z,\mathcal A)$ is transverse to $\Ps_t^{-1} (W,\mathcal B)$ in $(\proj^n,\mathcal S)$.  
 \end{cor}

In a recent paper \cite{MPintersectionhomology18} Corollary \ref{cor:transversality} is used to define an intersection pairing for \emph{real intersection homology}, an analog of intersection homology for real algebraic varieties.


\section{Construction of equisingular deformations}\label{sec:construction}

Let $f$ be either a polynomial or the germ of an analytic function, and let $V=V(f)$ denote the zero set of $f$.  We explain below how to construct Zariski equisingular deformations of $V$ (or more precisely of its equation $f$).  The idea comes from \cite{mostowski84}, where the local complex analytic case was considered.  
We begin with the global polynomial case as considered in \cite{PRpreprint18} since it is conceptually simpler and does not require Artin approximation.  Note also that this method can be applied as well to construct equisingular deformations of sets given by a system of several  equations 
as was explained at the end of section \ref{ssec:arbitrarycodimension}.

\subsection{Global polynomial case}\label{ssec:deform-polyn}

Given an algebraic subset $V$ of $\K^n$ and let the polynomials $g_1$,\ldots, $g_s\in\K[x]$ generate the ideal defining $V$.  
Let 
$$g_i=\sum_{\a\in\N^n}g_{i,\a}x^\a.$$ 
In general, a deformation of the $g_{i,\a}$, even arbitrarily 
small, destroys the topological structure of $V$ due to the presence of singularities (and "singularities at infinity" in the global case).  
In this method we construct a finite number of constraints 
satisfied by the coefficients $g_{i,\a}$, these are the equations \eqref{eq:2}, \eqref{eq:3}, 
\eqref{eq:5}, \eqref{eq:rational-relations} and the inequations \eqref{eq:cond-c_i} and \eqref{eq:4} below, that satisfy the following property.  
Any deformation $t\mapsto g_{i,\a}(t)$ with $g_{i,\a}(0)=g_{i,\a}$,  that satisfies the same constraints \eqref{eq:cond-c_i} and \eqref{eq:2}-\eqref{eq:rational-relations} is, by construction, Zariski equisingular. In particular any such deformation is topologically trivial.   
Moreover, the entries of \eqref{eq:cond-c_i} and \eqref{eq:2}-\eqref{eq:rational-relations} are rational functions in $g_{i,\a}$ with rational coefficients, that is they belong to $\Q(u_{i,\a})$, for some  new indeterminates $u_{i,\a}$. 

Let us fix a finite set of coefficients $g_{i,\a}\in \K$  that contains all 
nonzero of them.  In what follows we will perturb only these coefficients and keep all the other equal to zero. 

 After a linear change with rational coefficients of coordinates $x$ we can assume that
\begin{align}\label{eq:g-r-polynomials}
g_i=c_i x_n^{p_i}+\sum_{j=1}^{p_i}b_{n-1,r,j}(x^{n-1})x_n^{p_i-j} =\sum_{\b\in\N^n}a_{n,i,\b}x^\b ,\quad  \forall i=1,\ldots, s,
\end{align}
with 
\begin{align}\label{eq:cond-c_i}
c_i\ne 0, \qquad i = 1, \ldots, s.
\end{align} 
By multiplying each $g_i$ by $1/c_i$ we can assume that $c_i=1$ for every $i$. Denote by $f=f_n$ the product of the $g_i$ and by $a_{n}$ the vector of coefficients $a_{n,i,\b}$. The entries of $a_n$ are rational functions in the 
original $g_{i,\a}$ (i.e. before the linear change of coordinates $x$) with rational coefficients, say
\begin{equation}\label{eq:lin_ch} 
a_n=A_n(g_{i,\a}),
\end{equation}
where $A_n=(A_{n,i,\b})_{i,\b}\in \Q(u_{i,\a})^{N_n}$ for some integer $N_n>0$.  
Let the integer $l_n$ be defined by 
$$\D_{n,l_n}(a_{n})\not\equiv 0 \text { and } \D_{n,l}(a_n)\equiv 0, \quad \forall l<l_n, $$
where $\D_{n,l}$ denotes the $l$-th generalized discriminant of $f_n$, 
see Appendix. 
After a new linear change of coordinates $x^{n-1}$ with rational coefficients,  $\D_{n,l_n}(a_{n})=e_{n-1}f_{n-1}$
with $e_{n-1}\neq 0$ and 
$$f_{n-1}=\sum_{\b\in\N^n}a_{n-1,\b}x^\b=x_{n-1}^{d_{n-1}}+\sum_{j=1}^{d_{n-1}}b_{n-2,j}(x^{n-2})x_{n-1}^{d_{n-1}-j}$$
for  some  constants $a_{n-1,\b}$ and polynomials $b_{n-2,j}$. 
We repeat this construction and define recursively a sequence of polynomials $f_j(x^j)$, 
monic in $x_j$, such that
\begin{align}\label{eq:2}
\D_{j+1,l_{j+1}}(a_{j+1})=e_j\left(x_{j}^{d_{j}}+\sum_{k=1}^{d_{j}}b_{j-1,k}(x^{j-1})x_{j}^{d_{j}-k}\right)=e_j\left(\sum_{\b\in\N^n} a_{j,\b}x^\b\right)=e_jf_j 
\end{align}
is the first non identically equal to zero generalized discriminant  of $f_{j+1}$ and $a_j$ denotes the vector of coordinates $a_{j,\b}$.
This way we get a system of equations  
\begin{align}\label{eq:3}
\D_{j+1,l}(a_{j+1})\equiv 0\ \ \forall l<l_{j+1},\end{align}
and inequations
\begin{align}\label{eq:4}
\qquad e_j\ne 0 , \end{align}
for $j=n, n-1, \ldots, k_0$, until we get 
\begin{align}\label {eq:5}
f_{k_0}=1   \text { for some }
k_0\geq 0.
\end{align}

By  \eqref{eq:cond-c_i},  \eqref{eq:lin_ch} and \eqref{eq:2} the entries of the $c_i$, $a_k$ and $e_j$ are rational functions in the $g_{i,\a}$ with rational coefficients, let us say
\begin{align}\label{eq:rational-relations}
c_i= C_i(g_{i,\a}), \ \  a_k=A_k(g_{i,\a}),\ \ e_j=E_j(g_{i,\a}),
\end{align}
for some $C_i\in\Q(u_{i,\a})$, $A_k\in\Q(u_{i,\a})^{N_k}$ and $E_j\in\Q(u_{i,\a})$.  
Thus \eqref{eq:cond-c_i} and \eqref{eq:2}-\eqref{eq:rational-relations} are equations and inequations, with rational coefficients, on the original coefficients $g_{i,\a}$. 

Let $\mathcal V$ be an open connected neighborhood of a point $\tt \in \K^l$ and let $\mathcal O_{\mathcal V}$ denote the ring of $\K$-analytic functions on 
$\mathcal V$. 
Suppose that $g_{i,\a}(t)\in \mathcal O_{\mathcal V}$, where $t\in \mathcal V$, satisfy $g_{i,\a}= g_{i,\a}(\tt)$. 
For $t\in\mathcal V$ and $ i=1,\ldots, s,$ we define 
\begin{align*}
\tilde g_i(t,x):=\sum_{\a\in\N^n}g_{i,\a}(t) x^\a .
\end{align*}
We claim that if the $g_{i,\a}(t)$  satisfy the identities and the inequations 
\eqref{eq:cond-c_i} and \eqref{eq:2}-\eqref{eq:rational-relations},   then the family $t\to \{\tilde g_1(t,x)= \cdots = \tilde g_s(t,x) = 0\}$ is topologically trivial for $t$ in a small neighborhood of $\tt$ in 
$\mathcal V$.  For this we construct a system $F_j (t, x^j)$ satisfying the assumptions of Theorem \ref{thm:algebraic}.  We set 
\begin{align*}
F_n(t,x)=\prod_{i=1}^s G_i(t,x), \text{ where }  G_i(t,x)=\sum_{\b\in\N^n}A_{n,i,\b}(g_{i,\a}(t)) x^\b,  i=1,\ldots, s,
\end{align*}
and $A_n=(A_{n,i,\b})_{i,\b}\in \Q(u_{i,\a})^{N_n}$ given in 
\eqref{eq:lin_ch}.   Similarly for $j=k_0,\ldots, n-1$ we set 
\begin{align}\label{eq:functions-F}
F_j(t,x^j)=\sum_{\b\in\N^j}A_{j,\b}(g_{i,\a}(t))x^\b.
\end{align}
Note that $G_i(\tt,x)$ concide with $g_i$ after the linear change of coordinates made during the construction.  
It is clear from the above construction that the family $(F_j(t,x^j))$ satisfies the assumptions of Theorem \ref{thm:algebraic}.  Let us summarize it in the following.


\begin{thm}\label{thm:deformation-polynomial}
Suppose that $g_{i,\a}(t)$  satisfy the identities and the inequations \eqref{eq:cond-c_i} and \eqref{eq:2}-\eqref{eq:rational-relations}.  Then  $F_n (t, x)$ defines a Zariski equisingular family with respect to the parameter $t$.  
\end{thm}

 \subsection{Application: Algebraic sets are homeomorphic to algebraic sets defined over algebraic number fields}\label{ssec:homeotonumberfields}

The following result was proven in \cite{PRpreprint18}.  

\begin{thm}\label{thm:homeotoalgebraicglobal}
Let $V\subset \K^n$ (resp. $V\subset \P_\K^n$) be an affine (resp. projective) algebraic set, where $\K=\R$ or $\C$. Then there exist an affine (resp. projective) algebraic set $W\subset \K^n$ (resp. $W\subset \P_\K^n$) and a homeomorphism $h:\K^n\lgw \K^n$ (resp. $h:\P_\K^n\lgw \P_\K^n$) such that:
\begin{enumerate}
\item[(i)] the homeomorphism $h$ maps $V$ onto $W$,
\item[(ii)] $W$ is defined by polynomial equations with coefficients in 
$\ovl\Q\cap \K$,
\item[(iii)] the variety $W$ is obtained from $V$ by a Zariski equisingular deformation. In particular the homeomorphism $h$ can be chosen semialgebraic and arc-analytic.
\end{enumerate}
\end{thm}



Suppose, as in the previous section, that  the ideal defining $V$ is generated by the polynomials $g_1$,\ldots, $g_s\in\K[x]$.  
In order to prove Theorem \ref{thm:homeotoalgebraicglobal} one constructs 
in \cite{PRpreprint18} 
a deformation $t\mapsto g_{i,\a}(t)$ of the coefficients $g_{i,\a}\in \K$ of the $g_i$ that 
preserves all polynomial relations over $\Q$ satisfied by these coefficients.  Therefore this deformation preserves the identities 
\eqref{eq:2}, \eqref{eq:3}, \eqref{eq:5} and \eqref{eq:rational-relations}. If it is sufficiently small the inequations \eqref{eq:cond-c_i}, \eqref{eq:4} are also preserved and, by Theorem \ref{thm:deformation-polynomial}, the deformation is equisingular in the sense of Zariski. 

This construction is particularly simple if the field extension $\kk$ of $\Q$ generated by the coefficients $g_{i,\a}$ is a purely transcendental extension of $\Q$.  For the general case we refer the reader to \cite{PRpreprint18}.  
Thus assume that $\kk=\Q(\tt_1,\ldots, \tt_r)$, where the $\tt_i\in\K$ are algebraically independent over $\Q$.  Then there are rational functions  
$g_{i,\a} (t) \in \Q(t)$, 
$t=(t_1,\ldots, t_r)$, such that  $g_{i,\a} =g_{i,\a} (\tt )$.  Let $\mathcal V$ be a neighborhood of $\tt= (\tt_1,\ldots, \tt_r)$ that does not contain the poles of the $g_{i,\a}(t) $.  Since  
$\tt_i\in\K$ are algebraically independent any polynomial relation with coefficients 
in $\Q$, satisfied by $g_{i,\a} =g_{i,\a} (\tt )$, is also satisfied by $
g_{i,\a} (t)$.  
In particular, $g_{i,\a} (t)$ satisfy the identities 
\eqref{eq:2}, \eqref{eq:3}, \eqref{eq:5} and \eqref{eq:rational-relations}
 as we wanted. 
Choose $\q\in (\ovl\Q)^r \cap \mathcal V$ sufficiently close to $\tt$. 
Then all $g_{i,\a} (\q ) \in \ovl\Q$. Therefore the family $(F_j)$, 
defined by \eqref{eq:functions-F},  satisfies the hypothesis of Theorem \ref{thm:algebraic} and the hypersurfaces $X_0:=\{F_n(\q,x)=0\}$ and $X_1:=\{F_n(\tt,x)=0\}$ are homeomorphic. 
Moreover, thus constructed homeomorphism maps every  component of $X_0$ defined by $G_i(\q ,x)=0$ onto the component of $X_1$ defined by $G_i(\tt ,x)=0$, as in Theorem \ref{thm:many-components}. This proves that the algebraic variety $V=\{g_1=\cdots =g_s=0\}$ is homeomorphic to the algebraic variety $\{G_1(\q ,x)=\cdots=G_s(\q ,x)=0\}$ defined by polynomial equations over $\ovl\Q$. 

A result analogous to Theorem \ref{thm:homeotoalgebraicglobal} in the local case, for singularities of analytic spaces or analytic functions was proven by G. Rond in \cite{rond18}. 

\begin{rem}
Note that, by the above proof, in the special case when $\kk$ is a purely transcendental extension of $\Q$, we may replace, 
in the statement of Theorem \ref{thm:homeotoalgebraicglobal}, $\ovl\Q$ by $\Q$ if 
$\K=\R$, resp. $\Q[i]$ if 
$\K=\C$.  
In general, this is an open problem, whether every algebraic variety 
is homeomorphic to a variety defined over $\Q$, resp. $\Q[i]$.  In \cite{teissierCRAS90} B. Teissier gave an example of a complex analytic surface singularity defined over $\Q (\sqrt 5)$, which is not Whitney equisingular to any singularity defined over $\Q$. 
\end{rem}

\begin{question}{1. Open problem}Is every complex algebraic variety homeomorphic to a variety defined over 
$\Q[i]$ ? 
Is every real algebraic variety homeomorphic to a variety defined over 
$\Q$ ? 
\end{question}

\begin{question}{2. Open problem}Is every complex analytic set germ homeomorphic to a set germ defined over $\Q[i]$ ? 
Is every real analytic set germ homeomorphic to a 
set germ defined over $\Q$ ? 
\end{question}

\subsection{Analytic case}\label{ssec:deform-analytic}

Suppose now that $V$ is the germ at the origin of an analytic subset of $\K^n$ and let $g_1$,\ldots, $g_s\in\K\{x\}$ generate the ideal defining $V$.  We describe below, following \cite{mostowski84}, the construction of Zariski equisingular deformations of $V$.  The main idea is similar to that of subsection \ref{ssec:deform-polyn}, that is to use the discriminants of subsequent linear projections to construct a system of "constrains", that is equations and inequations satisfied by the $g_{i}$. These are the equations and inequations \eqref{eq: polynomials:f_i}, \eqref{eq:discriminants_i} defined below.  Then any deformation of the 
$g_{i}$ that satisfies the same constraints is Zariski equisingular.  
The main difference comes from the fact that now we are not going to use the  coefficients of the $g_{i}$, since there are infinitely many of them.  Instead we treat the equation of  \eqref{eq: polynomials:f_i}, \eqref{eq:discriminants_i}, as a system of equations on the functions $u_i(x^i), a_{i,j}(x^i)$, that is the coefficients of these  subsequent discriminants.

Let us consider a finite set of  distinguished polynomials 
$g_1, \ldots, g_s \in\K\{x\}$:  
   \begin{align*}
g_{i} ( x)= x_n^{r_i}+ \sum_{j=1}^{r_i} a_{n-1,i,j} 
(x^{n-1}) x_n^{r_i-j} ,
\end{align*}
i.e. we suppose $a_{n-1,i,j} (0) =0$ for all $i,j$. Arrange 
$a_{n-1,i,j}$ in a row vector  
$a_{n-1} \in \K\{x^{n-1}\}^{p_n}$, where $p_n:=\sum_i r_i$.  
Let $f_n$ be the product of the $g_i$'s.  The generalized discriminants 
$\D_{n,i} $ of $f_n$ are 
polynomials in the entries of $a_{n-1}$.  
  Let $l_n$ be a positive integer such that   \begin{align}\label{discriminants:n}
\D_{n,l} ( a_{n-1} )\equiv 0 \qquad l<l_n   ,
\end{align}
and  $\D_{n,l_n}  ( a_{n-1} ) \not \equiv 0$.  
 Then, after a linear change of coordinates $x^{n-1}$, by the Weierstrass Preparation Theorem, we may write 
    \begin{align*}
 \D_{n,l_n} ( a_{n-1} ) =  u_{n-1} (x^{n-1}) \Big (x_{n-1}^{p_{n-1}}
 + \sum_{j=1}^{p_{n-1}} a_{n-2,j} (x^{n-2}) x_{n-1}^{p_{n-1}-j} \Big ) 
 .
\end{align*}
where $u_{n-1}(0)\ne 0$ and for all $j$, $a_{n-2,j}(0)=0$.  We denote $$
f_{n-1} = x_{n-1}^{p_{n-1}}+
  \sum_{j=1}^{p_{n-1}} a_{n-2,j} ( x^{n-2}) x_{n-1}^{p_{n-1}-j} $$
    and the vector 
of its coefficients $a_{n-2,j}$ by $a_{n-2} \in \K\{x^{n-2}\}^{ p_{n-1}}$.   
Let $l_{n-1}$ be the positive integer such that the first $l_{n-1}-1$ generalized discriminants $\D_{n-1,l} $ 
of $f_{n-1}$ are identically zero and  $\D_{n-1,l_{n-1}} $ is not.  Then again we define 
$f_{n-2} ( x^{n-2})$ as the Weierstrass polynomial associated to 
 $\D_{n-1,l_{n-1}}  $. 

We continue this construction and  
 define a sequence of pseudopolynomials $f_{i} (  x^i )$, $i=1, \ldots, n-1$, such that 
 $f_i= x_i^{p_i}+ \sum_{j=1}^{p_i} a_{i-1,j} (x^{i-1}) x_i^{p_i-j}  $ is the Weierstrass polynomial associated to the first non-identically zero generalized discriminant $\D_{i+1,l_{i+1}} ( a_{i} )$ of $f_{i+1}$, 
where  we denote in general $a_{i}= (a_{i,1} , \ldots , a_{i,p_{i+1}} )$, 
  \begin{align}\label{eq: polynomials:f_i}
 \D_{i+1,l_{i+1}} ( a_{i} ) =  u_{i} (x^{i})  \Big (x_i^{p_i}+ \sum_{j=1}^{p_i} a_{i-1,j} (x^{i-1}) x_i^{p_i-j}  \Big) ,  
 \quad i=0,...,n-1 .
\end{align}
 Thus, for $i=0,...,n-1,$ the vector  
of functions $a_i$ satisfies  
  \begin{align}\label{eq:discriminants_i}
\D_{i+1,l} ( a_{i} )\equiv 0 \text { for } l<l_{i+1}   ,  \quad \D_{i+1,l_{i+1}} ( a_{i} ) \ne 0.
\end{align}

This means in particular that 
  \begin{align*}
\D_{1,k} ( a_{0} ) \equiv 0 \quad \text {for } l<l_1  \text { and }  \D_{1,l_1} ( a_{0} ) \equiv u_0 ,
\end{align*}
where $u_0$ is a non-zero constant.  

The following theorem follows from the construction of the family 
$u_i(t, x^i)$, $a_{i,j}(t,x^i)$.

\begin{thm}\label{thm:deformation-analytic}
Suppose that we extend all function $u_i(x^i), a_{i,j}(x^i)$ to analytic families $u_i(t, x^i), a_{i,j}(t,x^i)\in \K\{t,x\}$, $u_i(0, x^i)= u_i(x^i), a_{i,j}(0,x^i)= a_{i,j}(x^i)$,  where $t\in \K^ l$ is considered as a parameter. If the identities and 
the inequations of \eqref{eq: polynomials:f_i}, \eqref{eq:discriminants_i} are still satisfied by these extensions $u_i(t, x^i), a_{i,j}(t,x^i)$ 
then the family $f_n(t,x)=0$ is Zariski equisingular.  
\end{thm}


\subsection{Application: Analytic set germs are homeomorphic to algebraic ones}\label{ssec:homeoanalytic-algebraic}

The problem of approximation of analytic objects (sets or mappings) by algebraic ones has a long history,  see e.g. \cite{bochnakkucharz84} and the bibliography therein. In particular, several results were obtained in the case of isolated singularities. 
The local topological algebraicity of analytic set germs, in the general set-up, was first established in \cite{mostowski84} by Mostowski.  Given an analytic set germ $(V,0) \subset (\K^n,0)$, Mostowski shows the existence 
of a local homeomorphism $\tilde h: (\K^{2n+1},0) \to (\K^{2n+1},0)$ such that, after the embedding 
$(V,0) \subset (\K^n,0) \subset (\K^{2n+1},0)$, the image $\tilde h(V)$ is algebraic.  
It is easy to see that Mostowski's proof together with  Theorem 2 of \cite{bochnakkucharz84} gives 
the following result.  

\begin{thm} \label{thm:homeotoalgebraic}  
Let  $\K = \R$ or $\C$.  
Let  $(V,0) \subset (\K^n,0)$ be an analytic germ.  
 Then there is a homeomorphism $h: (\K^n,0) \to (\K^n,0)$ such that $h(V)$ 
 is the germ of an algebraic subset of $\K^n$.  
\end{thm}

We remark that in \cite {mostowski84} Mostowski states  his results 
only for $\K=\R$ but his proof also works for $\K=\C$.

The proof of Theorem \ref{thm:homeotoalgebraic} is, in principle, 
similar to the one of Theorem \ref{thm:homeotoalgebraicglobal}, but is techniquely much more demanding.   The main idea is to use Theorem \ref{thm:deformation-analytic} and deform analytic solutions 
of \eqref{eq: polynomials:f_i} and \eqref{eq:discriminants_i} 
to algebraic ones. Here by algebraic solutions we mean given by the ones 
defined by algebraic power series (an algebraic power series is a power series algebraic over $\K[x_1,...,x_n]$ - for example the power series 
$u(x)$ such that $u(0)=1$ and $u(x)^2=1+x$). Recall that the Artin approximation theorem states that convergent power series solutions of algebraic equations can be approximated by algebraic power series solutions.  
Clearly, we need a stronger result, not only an approximation but also a parameterized deformation from the old, convergent solutions to the new, algebraic  ones.  This is provided 
by P{\l}oski's  version of Artin approximation, see \cite{ploski74}. 
Finally, in order to apply Theorem \ref{thm:deformation-analytic} we need the nested Artin approximation, i.e. solutions  
$u_i(t,x^i), a_{i,j}(t,x^i)\in \K\{t,x^i\}$, of \eqref{eq: polynomials:f_i} and \eqref{eq:discriminants_i}, that depend only on  $x_1, \ldots, x_i $ 
and not on $x_k$ for $k>i$.  Nested Artin Approximation Theorem follows from  the N\'eron Desingularization, proven by Popescu \cite{popescu86}, and was not available at the time Mostowski's paper \cite {mostowski84} was written.  Instead, Mostowski proposes a recursive construction 
 of the system of equations \eqref{eq: polynomials:f_i} and \eqref{eq:discriminants_i} giving Zariski equisingularity conditions by local linear changes of coordinates and, 
  at the same time, step by step,  provides the deformation-approximation by algebraic power series solutions following the recipe given in \cite{ploski74}.  

One may shorten significantly Mostowski's construction using a stronger result, the nested variant of P{\l}oski's version Artin Approximation.  This is done in \cite{BPR17}, where
such Nested Artin-P{\l}oski-Popescu Approximation Theorem is proven. 
This theorem was used in \cite{BPR17} to deform $u_i(x^i), a_{i,j}(x^i)$ to algebraic power series solutions of \eqref{eq: polynomials:f_i} and \eqref{eq:discriminants_i}.  Furthermore,  a result of Bochnak-Kucharz \cite{bochnakkucharz84}, 
based on  Artin-Mazur Theorem of \cite{artinmazur65},  allows one 
to approximate the zeros of algebraic power series (or equivalently germs of Nash functions) by the zeros of polynomial functions. 

A stronger version of Theorem \ref{thm:homeotoalgebraic} was given in \cite{BKPR18} where it was shown that such a homeomorphism $h$ can be found with any prescribed order of tangency at the origin. \\

\begin{question}{3. Open problem}{What is the best level of regularity of homeomorphisms for which the statement of Theorem \ref{thm:homeotoalgebraic} holds ? It is known, for instance,  that Theorem \ref{thm:homeotoalgebraic}  is no longer true if one replaces "homeomorphism" by "diffeomorphism", for examples see \cite{bochnakkucharz84} and the last section of \cite{BPR17}. It is not known whether  Theorem \ref{thm:homeotoalgebraic} holds true if one requires the homeomorphism $h$ to be bi-Lipschitz.} 
\end{question} 

\subsection{Equisingularity of function germs.} \label{ssec:functions}

Zariski Equisingularity can also be used to construct topologically trivial deformations of analytic map germs, see  \cite{varchenkoizv72}. Let us consider first the case of functions as studied in \cite{BPR17},  that is the mappings with values in $\K$.  
Given a family $g_t(y) = g(t, y_1, \ldots , y_{n-1})$ of such germs parameterized by $t\in (T,t_0)$  We consider the associated family of set germs defined by the graph of $g$, the zero set of 
$F(t,x_1, \ldots, x_n) := x_1 - g(t, x_2, \ldots , x_{n})$, and construct a topological trivialization $h_t$ of $V=V(F)$ that does not move the variable $x_1$
\begin{align}\label{eq:preservex_1}
h_t (x_1, \ldots x_n) = (x_1 , \hat h_t (x_1, x_2, \ldots, x_n))  
\end{align}
so that $V\ni (t_0,x)$ if and only if $(t, h_t (x)) \in V$.  
Set $\sigma_t (y) := \hat h_t (g(y), y) $.  
Then 
$$
g_t\circ \sigma_t = g_{t_0}, 
$$
that is $g_t$ and $g_{t_0}$ are right (i.e. by a homeomorphism of the source) topologically equivalent.  Moreover, since $\sigma_t$ depends continuously on $t$ the family $g_t$ is topologically trivial.  

We now follow the main ideas of \cite{BPR17}  in order to explain the construction of topological trivialization of a family $V_t$ of analytic subspaces of $(\K ^n,0)$ that preserves the variable $x_1$. 
 For this we adapt the definition of Zariski equisingular families,  Definition \ref{def:ZAinfamilies}, by changing it slightly, and also 
 by changing accordingly the construction of equisingular deformations.   
The point is that, when we make linear changes of coordinates in order to replace a function by its Weierstrass polynomial, now we are no longer  allowed to change the variable $x_1$ and mix it with the other variables.  
So if one of the subsequent discriminants is divisible by $x_1$ we cannot proceed the way we have done it before.   Therefore we replace the assumptions (2) and (3) of Definition \ref{def:ZAinfamilies} by 
 \begin{enumerate}
\item [2'.]
There are $q_i\in \N$ such that the discriminant of $(F_{i})_{red}$ divides $x_1^{q_i} F_{i-1} (t, x^{i-1})$. 
\item [3'.]
$F_1\equiv 1$. 
\end{enumerate}

Then the construction of the homeomorphisms that we presented in Section \ref{ssec:topequising} gives the following version of Theorem \ref{thm:topological_triviality}, that is a simplified statement of 
\cite[Theorem 5.1]{BPR17}.

\begin{thm}\label{thm:triv-functions} 
Suppose that $V$ is Zariski equisingular with respect to the parameter $t$ in the sense of  Definition \ref{def:ZAinfamilies} with the conditions 2 and 3  replaced by conditions 2'and 3'.  Then we may require that the homeomorphisms $\Ph$ of \eqref{eq:trivialization} satisfies 
additionally  $\Ps_1(t, x_1) =x_1$. 
\end{thm}

\begin{proof} 
Idea of proof. 

Because $F_1\equiv 1$, by 2', the discriminant locus of $F_2$ is either empty or given by $x_1=0$.  
Therefore we may take $\Ps_1(t, x_1) =x_1$.  
Then we show by induction on $i$  that each $\Ph_i$ 
can be lifted so that the lift $\Ph_{i+1}$  preserves the zero set of 
$F_{i+1}$ and the values of $x_1$.  The former 
condition follows by inductive assumption and the fact that $\Ph_i$ preserves the discriminant locus of $F_{i+1}$.  
The latter condition is satisfied trivially since $\Ph_{i+1}$ 
is a lift of $\Ph_{i}$.  
\end{proof}

 As a corollary we obtain the following result.

\begin{thm} [{\cite[Theorem 1.2]{BPR17}}]
\label{thm:homeotopolynomial}  
Let  $\K = \R$ or $\C$.  
 Let  $g: (\K^n,0)\to (\K, 0)$ be an analytic function germ.  
 Then there is a homeomorphism $\sigma : (\K^n,0) \to (\K^n,0)$ such that $g\circ \sigma$ 
 is the germ of a polynomial.  
\end{thm}

For the proof of Theorem \ref{thm:homeotopolynomial} one follows the 
proof of Theorem \ref{thm:homeotoalgebraic} that gives such homeomorphism to a Nash function and not directly to a polynomial, since we cannot get a better result just using the Artin approximation.  
Recall that a function is \emph{Nash} if it is analytic and satisfies an algebraic equation. Thus $f:(\K^n,0) \to \K$ is the germ of a Nash function if and only if its Taylor series 
is an algebraic power series. For more details on real and complex Nash functions and sets see \cite{BCR}, \cite{bochnakkucharz84}. 
The final step of the proof of Theorem \ref{thm:homeotopolynomial}, a homeomorphism of a Nash germ to a polynomial germ follows from 
\cite{bochnakkucharz84}, that is in essence from the Artin-Mazur Theorem, 
and a Thom stratification argument, see section 5.5 of \cite{BPR17} for details.

There is a common generalization of Theorems \ref{thm:homeotopolynomial} and \ref{thm:homeotoalgebraic}.

\begin{thm} [{\cite[Theorem 1.3]{BPR17}}] \label{thm:generalhomeo}   
Let  $(V_i,0) \subset (\K^n,0)$ be a finite family of analytic set germs and let  $g: (\K^n,0)\to (\K, 0)$ be an analytic function germ.  Then there is a homeomorphism $\sigma : (\K^n,0) \to (\K^n,0)$ such that $g\circ \sigma$ 
 is the germ of a polynomial, and  for each $i$, $\sigma^{-1} (V_i)$ 
 is the germ of an algebraic subset of $\K^n$.  
\end{thm}

Theorem \ref{thm:generalhomeo} cannot be extended to many functions or to maps with values in $\K^p$ for $p > 1$, see  \cite[Example 6.3]{BPR17}. 

\begin{cor}[{\cite[Corollary 1.4]{BPR17}}]
Let  $g: (V,p)\to (\K,0)$ be an analytic function germ defined on the germ $(V,p)$ of an analytic space.  Then there 
exists an algebraic affine variety $V_1$, a point $p_1\in V_1$, the germ of a polynomial function $g_1:(V_1,p_1)\to 
(\K,0)$ and  a homeomorphism $\sigma : (V_1,p_1) \to (V,p)$ such that $g_1= g\circ \sigma$.  
\end{cor}

We do not know whether the above results hold true with "homeomorphism" replaced  by "bi-Lipschitz homeomorphism". 

\begin{question}{Open problem}{4. Is an analytic function germ bi-Lipschitz homeomorphic to a Nash or a polynomial germ ? }
\end{question} 

Unlike the analogous open problem for analytic set germs, it is more likely that the answer to this one is negative.  The reason is the following.  
By the existence of Lipschitz stratification, cf. \cite{mostowski85}, \cite{PA94}, the bi-Lipschitz equivalence of analytic set germs does not have continuous moduli (the principle of generic bi-Lipschitz triviality, analogous to the one described in subsection 
\ref{ssec:arbitrarycodimension}, holds true).  On the other hand, the bi-Lipschitz right equivalence of analytic function germs admits continuous moduli, see \cite{henryparus2003}, \cite{henryparus2004}.  

Using Theorem \ref{thm:triv-functions} one may show that the principle of generic topological equisingularity of analytic function germs holds true. 
(An alternative proof follows again by stratification theory, more precisely by Thom stratification and Thom-Mather Isotopy Lemma.)
Let $T$ be a $\K$-analytic space and let $g_t(y)=g(t, y) : (T,t_0)\times (\K^n,0) \to (\K ,0)$ be a $\K$-analytic family of $\K$-analytic function germs. 
We say that the family $g_t(y)$ is \emph{topologically trivial at $t_0$} (for topological right equivalence) if there are an open  neighborhood $T' $ of $t_0$ in $T$ and neighborhoods $\Omega_0$ of the origin in $\K^n$, and $\Omega$ of $(t_0,0)$ in $T\times \K^{n}$, and  a 
homeomorphism 
\begin{align*}
\Ph  : U \times \Omega_0 \to \Omega, 
\end{align*}
 such that 
 $g(\Ph (t,y))=g(0,y)$. Then the following statement holds.

 \begin{cor}[Principle of generic topological equisingularity of analytic function germs,
{\cite[Theorem 8.5]{PP17}}]\label{thm:gen-top-equising-functions}
Let $T$ be a $\K$-analytic space and let $g_t(y) : (T,t_0)\times (\K^n,0) \to (\K ,0)$ be a $\K$-analytic family of $\K$-analytic function germs.  Let $t_0\in T$.   
Then there exist an analytic stratification of an open neighborhood of $t_0$ in $T$ such that for every stratum $S$ and every $t'_0\in S$, the family 
$g_t(y), t\in S$ is topologically trivial at $t_0'$. 
\end{cor}

\subsection{Local topological classification of smooth mappings.}\label{ssec:smoothmappings}
The principle of generic topological equisingularity does not hold for the germs of mappings.  That is it is known by an example of Thom \cite{thom62}, see also \cite{nakai84},   that the topological classification of real or complex, analytic or polynomial map germs admits continuous moduli.  This means that there are, polynomial in $t$, families of polynomial map germs $f_t:(\K^n,0)\to (\K^p,0)$ that have different topological types for  different $t$,  provided $n\ge 3, p\ge 2$, see \cite{nakai84}. 
Recall that we say that two germs $f_i:(\K^n,0)\to (\K^p,0)$, $i=1,2$, 
have \emph{the same topological type} if there exist homeomorphisms germs 
$h:(\K^n,0)\to (\K^n,0)$ and $g:(\K^p,0)\to (\K^p,0)$ such that 
$f_1\circ h= g\circ f_2$, that is, in other words, they are right-left topologically equivalent.

A smooth map germ $f:(\R^n,0)\to (\R^p,0)$ is \emph{topologically $r$-determined} if every smooth map germ with the same $r$-jet as $f$ is topologically equivalent to $f$. In \mbox{
\cite{thom64} }
Thom proposed a stabilization theorem: For any positive integer $r$, there is a closed semialgebraic subset $\Sigma_r$ of the $r$-jet space $J_r(n,p)$ such that 
(i) 
$\codim \Sigma_r \to \infty$ as $r\to \infty$, and (ii) if the $r$-jet of a map-germ $f$ belongs to $J_r(n,p) \setminus \Sigma_r$, then $f$ is topologically $r$-determined. In other words "most" smooth mappings, that is up to a set of infinite codimension in the jet space, look algebraic and are finitely determined.  Thom gave a sketch of proof in \cite{thom64}.  The first complete proof was given by Varchenko in 
\cite{varchenkoizv73,varchenkoizv74} 
using very different ideas that the ones of Thom, namely Zariski equisingularity.  Actually, Varchenko proved a much stronger result.

\begin{thm}[{\cite[Theorem 2]{varchenkoICM75}}]\label{thm:finiteldetermined}
There exists a partition of the space of $r$-jets $J_r(n,p)$ in disjoint semialgebraic sets $V_0, V_1, \ldots$ having the following properties. 
\begin{enumerate}
	\item
Maps whose jets live in the same $V_i$, $i>0$, are (right-left) topologically equivalent.
\item
Any germ whose $r$-jet is in $V_i$ for $i>0$ is simplicial for suitable triangulations of 
$\R^n$ and $\R^p$.
\item
The codimension of $V_0$ in $J_r(n,p)$ tends to infinity as $r$ tends to infinity.
\end{enumerate}
\end{thm}

The stabilization theorem of Thom was also shown by du Plessis in 
\cite{duplessis82}.  The proof given there follows the original Thom's ideas, stratification theory, transversality, isotopy lemmas and  Mather's ideas about versal unfoldings.
Another application of Zariski equisingularity method to finite determinacy was given in 
\cite{bobadilla2004}, where the function case ($p=1$) was considered.  Note that topologically finitely determined function germs 
$f:(\K^n,0)\to (\K,0)$ have isolated singularities (or are regular).  
In \cite{bobadilla2004} Bobadilla gives a meaningful version of Theorem 
\ref{thm:finiteldetermined} for non-isolated singularities by considering functions belonging to a fixed ideal $I$ instead of the whole space of analytic germs at the origin.  We refer the reader to 
\cite{bobadilla2004} for details.  

\begin{rem}
If the target space of $f_t$ is of dimension bigger than one then the method, we applied in the previous subsection to trivialize families of function germs may not work.  In general we cannot trivialize the family of graphs starting from the variables in the target, as we did by taking $F(t,x_1, \ldots, x_n) := x_1 - g(t, x_2, \ldots , x_{n})$, if this graph is not included in the zero set of a Weierstrass polynomial in a variable in the source.  
This is related to the presence of fibers of dimension bigger than the expected dimension (dimension of the source minus dimension of the target), not only for $f$ but for every function (discriminant) obtained during the construction process.

Even if this phenomenon of "blowing-up" of the special fiber is not present, that is we can apply Zariski method without mixing the variables of the source and of the target, we cannot, in general, construct topological trivialization that is the identity on the target.  That means that, if $p\ge 2$, we get the right-left equivalence instead of the right one as in Theorem \ref{thm:gen-top-equising-functions}.  
\end{rem}


\section{Equisingularity along a nonsingular subspace.  
Zariski's dimensionality type}\label{sec:ZEalong}

In \cite[Definition 3] {zariski71open}, see also \cite{varchenkoICM75},  Zariski introduced the notion of algebro-geometric equisingularity, now called Zariski equisingularity, of an algebroid hypersurface $V\subset \K^{r+1}$ along a nonsingular subspace of $\Sing V$.  This notion can be easily adapted to the complex and real analytic set-ups.

Let $V=f^{-1}(0)$ be an analytic hypersurface defined in a neighborhood of a point $\point \in \K^{r+1}$. As before we assume that $f$ is reduced. Let 
$\W$ be a nonsingular analytic subspace of $\Sing V$ containing $\point$. 
Let $x_1,x_2, \ldots, x_{r+1}$ be a local coordinate system at $\point$. 
Consider a set of $r$ elements $z_1, z_2, \ldots ,z_r$ of the local ring of $V$ at $\point$ :
$$
z_i = z_i(x_1, x_2, \ldots, x_{r+1}) = z_{i,1} + z_{1,2} + \cdots , \quad 
i = 1, 2, \ldots, r,
$$
where the $z_i$ are convergent power series in the $x$'s, and $z_{i,\alpha}$ is homogeneous of degree $\alpha$.
We say that the $r$ elements $z_i$ form \emph{a set of parameters} if the following two conditions are satisfied :
\begin{enumerate}
\item [(a)]
\emph{$x = 0$ is an isolated solution of the $r + 1$ equations $z_1(x) = z_2(x) = \cdots = z_r(x) = f(x) = 0$.
\item [(b)]
The $ r$ linear forms $z_{i,1}$ are linearly independent.}
\end{enumerate}
If condition (b) is satisfied, then the $r$ linear equations
$z_{i,1} (x) = 0, i = 1, 2, \ldots , r$ define a line $l_z$ through $P$ and 
the parameters $z_i$ define a co-rank one projection $\pi_z $ of a neighborhood of 
$\point$ in $\K^{r+1}$ onto a neighborhood of $\bpoint= \pi (\point)$ 
in  $\K^{r}$.  
This projection $\pi_z$ is called \emph{permissible} if the fiber $\pi^{-1}(\pi (\point ))$, that is a nonsingular curve, is transverse to $\W$ (here 
"transverse"  means that the the tangent line to the fiber is not tangent to $W$).  
If this curve is transverse to $V$ at $\point$, that is $l_z$ does not belong to the tangent cone $C_\point (V)$ to $V$ at $\point$, then the projection is called \emph{transversal} (or transverse) and the $z_i(x), i=1, \ldots, r$, are \emph{transversal parameters}.  

Let $\pi_z$ be a permissible projection and let $\pi_{z,V}$ denote the restriction of $\pi$ to $V$. 
Thus we may suppose that locally $f$ is a suitable reduced Weierstrass polynomial whose discriminant $\D_f$ is an analytic function in 
$(z_1, z_2, \ldots ,z_r)$.  Denote by $\DD_z$ its zero set, that is the discriminant locus of $\pi_{z,V}$.  

The projection $\pi_z (\W)$ is a nonsingular variety $\bW$, of the same dimension as $\W$. 
Since we have assumed that $W\subset \Sing V$ we have $\bW\subset \DD_z$.  
If $\dim \W = \dim \DD_z =r-1$ then we say that $V$ is \emph{Zariski equisingular at $\point $ along $\W$} if $\bpoint$ is a non-singular point of $\DD_z$.  In the general case, Zariski's definition is inductive and goes as follows.

\begin{defn}\label{def:ZAalong}
We say that $V$ is \emph{Zariski equisingular at $\point $ along $\W$ } if there  exists a permissible projection $\pi_z$ such that 
$\DD_z$ is Zariski equisingular along $\bW$ at $\bpoint$ 
(or if $\bpoint$ is a nonsingular point of $\DD_z$). 
\end{defn}


\subsection{Equimultiplicity.  Transversality of projection}\label{ssec:equimultiplicity}

As Zariski states on page 489 of \cite{zariski71open} the algebro-geometric equisingularity,  i.e. Zariski equisingularity as defined 
in Definition \ref{def:ZAalong}, implies equimultplicity.

\begin{prop}\label{prop:equimultiplicity}
If $V$ is Zariski equisingular at $\point $ along $\W$ then the multiplicity 
of $V$ is constant along $\W$. 
\end{prop}

Zariski proves it when $\dim \W = \dim V-1=r-1$, see \cite[Theorem 7]{zariski65-S1}, and in the general algebroid case in \cite{zariski75}.  
For a proof in the complex, and also real analytic case, see 
\cite[Proposition 3.6]{PP17}. 


Similarly to Definition \ref{def:ZAalong} one may 
define  \emph{transverse Zariski equisingularity along a nonsingular subspace} as the one given by 
transverse projections. By Proposition \ref{prop:equimultiplicity}, because the equimultiple families are normally pseudo-flat 
(continuity of the tangent cone),  the transversality of $\pi_z$ at $\point$ implies the transversality at all points of $\W$ in a neighborhood of $\point$. 

One can also define  \emph{generic or generic linear Zariski equisingularity along a nonsingular subspace}.  For generic linear it means that we require at each stage the projection to be chosen from a Zariski open non-empty set of linear projection. Note that a priori this notion depends on the choice of coordinates and it is not clear whether it is preserved by nonlinear changes of coordinates.  We discuss the notion of  generic projection in subsections \ref{ssec:motivation} and \ref{ssec:dimensionality}. 

\subsection{Relation to other equisingularity conditions}\label{ssec:relationto}

As we mentioned before Varchenko showed in \cite{varchenkoizv72}, see also 
\cite{varchenkoizv72,varchenkoICM75},  that in the complex or real analytic case Zariski equisingularity implies topological triviality.  

In  \cite[Question E]{varchenkoizv72},   Zariski asked as well whether Zariski equisingularity implies Whitney's conditions.  This has been disproved by 
Brian{\c c}on end Speder in \cite{brianconspeder75a} for the equisingularity as defined in Definition \ref{def:ZAalong}.  
In \cite{speder75} Speder shows that if $V$ is Zariski equisingular along 
a nonsingular variety $\W$ for sufficiently generic projections, then the pair $(\Reg (V), \W)$ satisfies Whitney's Conditions.  For instance, generic linear projections, that is from a Zariski open non-empty subset of such projections, are generic in the sense of Speder. This result was improved in \cite{PP17}, where it was shown that  transverse Zariski equisingularity, both in real and complex analytic cases,  implies Whitney's Conditions, see Theorem 4.3 and  Theorem 7.1 of \cite{PP17} for precise statements.


There are several classical examples describing the relation between 
Zariski equisingularity and Whitney's conditions.  The general set up 
for these examples is the following.  Consider a complex algebraic hypersurface 
$X\subset \C^4$ defined by a polynomial $F(t,x,y,z)=0$ such that $\Sing X =T$, where $T$ is the $t$-axis. Let $\pi : \C^4 \to T$ be the standard projection.  In all these examples $X_t =\pi ^{-1} (t)$, $t\in T$, is a family of  isolated singularities,  topologically trivial along $T$.  These examples relate the following conditions:
 \begin{enumerate}
 \item
 $X$ is Zariski equisingular along $T$,  Definition  \ref{def:ZAalong}.  
 \item 
  $X$ is Zariski equisingular along $T$ for a transverse projection.  
  \item
    $X$ is Zariski equisingular along $T$ for a generic system of coordinates.  Here we consider "generic" in the sense of  \cite{brianconhenry80}.  It is equivalent, see loc. cit. to be generic linear, 
    or generic in the sense of Zariski \cite{zariski1979}, that we recall in Section \ref{ssec:dimensionality} below. 
    \item
    The pair $(X\setminus T,T)$ satisfies Whitney's conditions (a) and (b). 
  \end{enumerate}

Clearly (3)$\Rightarrow$(2)$\Rightarrow$(1).  Speder showed  (3)$\Rightarrow$(4) in \cite{speder75}  
and (2)$\Rightarrow$(4) 
for families of complex analytic hypersurfaces with isolated singularities in $\C^3$  in his thesis 
\cite{spederthese} (not published). 
Theorem 7.1 of \cite{PP17}
 gives (2)$\Rightarrow$(4) in the general case.   As the examples below show, all the other 
 implications are false.  

\begin{example}[\cite {brianconspeder75a}]\label{ex:briancon-speder}
\begin{align}
F(x,y,z,t)= z^5 + t y^6 z + y^7 x + x^{15} 
\end{align}
This example satisfies (1) for the projections  $(x,y,z) \to (y,z) \to x$ but  (4) fails.  
As follows from  Theorem 7.1 of \cite{PP17}, (2) fails as well.  
\end{example}

\begin{example}[\cite {brianconspeder75b}]
\begin{align}
F(x,y,z,t)= z^3 + t x^4 z + y^6 + x^{6} 
\end{align}
In this example (4) is satisfied and (3) fails.  
This example satisfies (1) for the projections  $(x,y,z) \to (x,z) \to x$.   
\end{example}

\begin{example}[\cite {luengo85}]\label{ex:luengo}
\begin{align}
F(x,y,z,t)= z^{16} + ty z^3 x^7 + y^6z^4+ y^{10} + x^{10} 
\end{align}
In this example (2) is satisfied and (3) fails.  
\end{example}

\begin{example}[\cite{parusinski1985}]
\begin{align}
F(x,y,z,t)= x^9 + y ^{12} + z^{15} + tx^3 y ^4 z^5 
\end{align}
In this example (4) is satisfied and (1) fails.  This also shows that (4) does not imply (2).
\end{example}


\subsection{Lipschitz equisingularity}\label{ssec:lipschitz}

In 1985 T. Mostowski \cite{mostowski85} introduced the notion of Lipschitz stratification and showed the existence of such stratification for germs of complex analytic subsets of $\C^n$. For complex algebraic varieties, such stratification exists globally. 
The existence of Lipschitz stratification for real analytic spaces and algebraic varieties was shown \cite{Lipschitz-Fourier}, \cite{PA94}.  
Lipschitz stratification satisfies the extension property of  stratified Lipschitz vector fields from lower-dimensional to higher-dimensional strata, and therefore implies local bi-Lipschitz triviality along each stratum (and hence Lipschitz equisingularity as well).  Mostowski's construction is similar to the one of Zariski, but involves considering many co-rank one generic projections at each stage of construction. For more on Lipschitz stratification, 
we refer the interested reader to \cite{mostowski85}, \cite{Lipschitz-review}, \cite{halupczok-yin2018}.  

By Lipschitz saturation, see \cite{PTpreprint}, an equisingular family of complex analytic plane curves is bi-Lipschitz trivial, i.e. trivial by a local ambient bi-Lipschitz homeomorphism.  In general, there is a conjectural relation between Lipschitz and Zariski equisingularity, at least in the complex analytic set up.

\begin{question}{5. Open problem}{Are generically Zariski equisingular families of complex hypersurfaces bi-Lipschitz equisingular?  Does Zariski equisingularity provide a "natural" way of construction of Lipschitz stratification in the sense of Mostowski ?  
} 
\end{question} 

For families of complex surface singularities, that is along a nonsingular subspace of codimension 2 the following results have been announced in \cite{NPpreprint}, \cite{PPpreprint19}.  In  \cite{NPpreprint} it was shown that generically Zariski equisingular families of normal complex surface singularities are bi-Lipschitz trivial.  In \cite{PPpreprint19} was shown that a natural stratification given by subsequent generic 
(or generic linear) projections of a complex hypersurface satisfies Mostowski's Conditions in codimension 2.  In particular, the latter result implies that generic Zariski equisingulariy of families (not necessarily isolated) of complex surface hypersurface singularities is Lipschitz equisingular.


\subsection{Zariski dimensionality type.  Motivation}\label{ssec:motivation}

When $\dim \W = \dim V - 1$,
 $V$ is Zariski equisingular at $\point $ along $\W$ 
if and only if $V$ is isomorphic to the total space of an equisingular family of plane curve singularities along $\W$, see \cite[Theorem 4.4]{zariski65-S2}.   Then, moreover, Zariski equisingularity can be realized by 
any transversal projection $\pi_z$.  

Guided by this example, Zariski conjectured in \cite[Question I]{zariski71open}, 
that Zariski equisingularity for a single permissible projection implies the equisingularity for generic projection (or for almost all projections that we recall later in section \ref{ssec:almostall}).   
An affirmative answer to this question would imply, in particular, that if there exists an equisingular projection then there exist a transversal equisingular projection. Both turned out not to be true.    
In \cite{luengo85} Luengo gave an example of a family of surface 
singularities in $\C^3$ that is Zariski equisingular for one projection, that is even transversal, but is not equisingular for the generic projection, see Example 
\ref{ex:luengo}. Brian\c con end Speder gave in \cite {brianconspeder75a}
an example that is equisingular for one projection but there is no transversal projection that gives Zariski equisingularity, see Example 
\ref{ex:briancon-speder}.  

Therefore Zariski  in \cite{zariski1979} proposes a different strategy.  Instead of arbitrary permissible projections, or even transversal  
projections, Zariski uses generic projections to define the equisingularity relation.  (We recall what "generic" means for Zariski in the next subsection.)  Having fixed such an equisingularity relation, Zariski introduces the notion of dimensionality type.  For this the equisingularity relation should first satisfy the following property.\\

\noindent
\emph{The set of points of equivalent singularities form a locally nonsingular subspace of $V$ of codimension that depends only on this equisingularity class.}
\\

Thus, for a point $\point \in V$ the set of points equivalent to $(V,\point)$ is nonsingular and its codimension in $V$ characterizes how complicated the singularity is.  This codimension is then called \emph{the dimensionality type of $(V,\point)$}.  The points of dimensionality type $0$ are the nonsingular points of $V$.  The simplest singular points of $V$, of dimensionality type $1$, are those at which $V$ is isomorphic to the total space of an equisingular family of plane curves. 
The closure of the set of points of fixed equisingularity type may contain 
points of different equisingularity type but only of the higher dimensionality type and only on finitely many such equisingular strata.

 The very definition of what is meant by the word "generic" is the main point of  Zariski's definition. Let us make a quick comment on an apparently obvious choice.  Similarly to Definition \ref{def:ZAalong} one may 
define  \emph{generic linear Zariski equisingularity} as the one given by linear projections belonging to a Zariski open non-empty subset of linear projections. 
Except the case of the dimensionality type $1$, it is not clear whether such notion of generic linear Zariski equisingularity is preserved by non-linear local changes of coordinates, nor whether it implies the generic Zariski equisingularity.  

\subsection{Zariski dimensionality type}\label{ssec:dimensionality}

Formally Zariski's original definition of the dimensionality type requires the field extension by infinitely many indeterminates.  Therefore, in this subsection exceptionally,  we work over an arbitrarily algebraically closed field of characterisitc zero that will be denoted by $\kk$,  
and instead of the category of complex analytic spaces we consider the category of algebraic or algebroid varieties.  Recall that the algebroid varieties are the varieties defined by ideals of the rings of formal power series, see \cite[Ch. IV]{lefschetz53}   and \cite[Section 2]{zariski1979}. 
We note, however, that in \cite[Proposition 5.3]{zariski1979} Zariski shows that his definition involving such a field extension can be replaced  by a condition that is based on the notion of almost all projections that does not require a field extension.  
We recall the approach via almost all projections in the next subsection. 

Let $\kk$ be an algebraically closed field of characteristic zero.
Consider an algebroid  hypersurface $
V=f^{-1}(0) \subset (\kk^{r+1},\point )$ at $\point \in \kk^{r+1}$ defined in a local system of coordinates by a formal power series $f\in \kk [[x_1, \ldots , x_{r+1}]]$ (that we assume  reduced). Zariski's definition of the dimensionality type is based on the following notion of generic projection.  \emph{The generic projection}, in the sense of \cite{zariski1979}, is  the map 
    $\pi_u (x) = (\pi_{u.1}(x), \ldots, \pi_{u,{r}}(x))$, with 
  \begin{align}\label{eq:genprojection}
\pi_{u,i} (x) = \sum_{d\ge 1} \sum_{\nu_1+ \cdots \nu_{r+1}=d} u^{(i)}_{\nu_1, \cdots, \nu_{r+1} } x^\nu . 
  \end{align}
This map is defined over $\kk^*$, any field extension of $k$ that contains all coefficients $u^{(i)}_{\nu_1, \cdots, \nu_{r+1} }$ 
as interdeterminates, thus formally 
$\pi_u: ((\kk^*)^{r+1},\point ) \to ((\kk^*)^{r},\point_0^*)$, where 
 $\point_0^*= \pi_u (\point)$. Denote by 
 $\DD_u ^*\subset ({\kk^*}^{r},\point_0^* )$ the discriminant locus of 
 $(\pi_u)_{|V^*}$, where $V^*=f^{-1}(0) \subset ({\kk^*}^{(r+1)},\point )$.

 Let $\W$ be a nonsingular algebroid subspace of $\Sing V$ and let 
 $\bbW = \pi_u (\W)$.  If $\dim \W =\dim V-1$ then we say that 
  $V$ is generically Zariski equisingular at $\point $ along $\W$ 
 if  $\bar \point^*$ is a non-singular point of $\DD_u^*$. In general, the definition is similar to Definition \ref{def:ZAalong}. 

\begin{defn}\label{def:generic equisingularity}
We say that $V$ is \emph{generically Zariski equisingular at $\point $ along $\W$ }
 if $\DD_u ^*$ is generically Zariski equisingular 
 along $\bbW$ at $\bar \point^*$ (or if $\bar \point^*$ is a non-singular point $\DD_u^*$).
\end{defn}  

The definition of dimensionality type of \cite{zariski1979} is again recursive. It is defined 
for any point $Q$ of $V$, not only the closed point $\point $. 

\begin{defn}\label{def:dimensionalitytype}
Any simple (i.e. non-singular)  point $Q$ of $V$ is  of dimensionality type $0$.  Let $Q$ be a singular point $V$ and let $Q^*_0= \pi_u (Q)$.  
Then \emph{the dimensionality type of $V$ at $Q$}, 
denoted by $\dimtype (V,Q)$, is equal to 
$$\dimtype (V,Q) = 1+ \dimtype (\DD^*_u,Q^*_0).$$  
\end{defn}

The  notions of generic Zariski equisingularity and of dimensionality type are independent of the choice of 
this field extension $\kk^*$, see  
\cite{zariski1979} and \cite{zariski1980}.

As follows from \cite{zariski1979} the set of points where the dimensionality type is constant, say equal to $\sigma$, is either
 empty 
or a nonsingular locally closed subvariety of $V$ of codimension $\sigma$.   
The dimensionality type defines a stratification $V= \sqcup_\alpha S_\alpha $ of $V$ that satisfies the frontier condition, i.e. if 
$S_\alpha \cap S_\beta \ne \emptyset$ then $S_\alpha \subset \overline {S_\beta}$, and 
$V$ is generically Zariski equisingular along $\W$ at $\point$ if and only if  $W$ is contained in the stratum containing $\point$.  

A  singularity is of dimensionality type $1$ if and only if it is isomorphic to the total space of an equisingular family of plane curve singularities, see \cite{zariski65-S2} Theorem 4.4. 
 Moreover, if this is the case, then 
$V$ is such an equisingular family for any transverse system of coordinates.  

\subsection{Almost all projections.}\label{ssec:almostall}
In  \cite[Proposition 5.3]{zariski1979} Zariski shows that, in  Definitions \ref{def:generic equisingularity} and \ref{def:dimensionalitytype}, the generic projection $\pi_u$ can be replaced by a condition that involves almost all projections $\pi_{\bar u}: \kk^{r+1}\to \kk^{r}$ (so it does not require a field extension).   

 One says that a property holds \emph{for almost all projections} if 
 there exists a finite set of polynomials $\mathcal G=\{G_\mu\}$ in the indeterminates  $u^{(i)}_{\nu_1, \cdots, \nu_{r+1} }$ and coefficients in $\kk$ such that 
 this property holds for all projections $\pi_{\bar u}$ for $\bar u$ 
 satisfying  $\forall_\mu G_\mu (\bar u)\ne 0 $.  
 Here the bar denotes the specialization $u \to \bar u$, i.e. 
 we replace all indeterminates $u^{(i)}_{\nu_1, \cdots, \nu_{r+1} }$ by 
 elements of $\kk$, $\bar u^{(i)}_{\nu_1, \cdots, \nu_{r+1} }\in \kk$.  Thus, for almost all  projections $\pi_{\bar u} $ 
 the dimensionality type $\dimtype (V,\point )$ equals 
 \begin{align}\label{eq:dimtype-polynproj} 
 \dimtype (V,\point ) = 1+ \dimtype (\DD_{\bar u},\pi_{\bar u} (\point )), \end{align}
  where $\DD_{\bar u}$ 
 denotes the discriminant locus of ${\pi_{\bar u}}_{|V}$.   Since the finite set of polynomials $\mathcal G$ involves nontrivially only finitely many indeterminates  $u^{(i)}_{\nu_1, \cdots, \nu_{r+1} }$, we may specialize the remaining ones to $0$, and then the projection 
 $\pi_{\bar u} $ becomes polynomial. 
 This means that, as soon as we know the set of polynomials $\mathcal G$, 
 we may compute the dimensionality type of $V$ at $\point $ 
 just by computing $\dimtype (\Delta_{\bar u},\pi_{\bar u} (\point ))$, 
 for only one polynomial projection $\pi_{\bar u} $, satisfying $\forall_\mu G_\mu (\bar u)\ne 0 $.  
 Similarly, in order to check whether  $V$ is generic Zariski 
 equisingular at $\point $ along $S$, 
 it suffices to check it for $\Delta_{\bar u} ^*$ 
 along $\pi_{\bar u} (S)$ at $\pi_{\bar u} (\point )$. 

\subsection{Canonical stratification of hypersurfaces}\label{ssec:canonstrat}

The dimensionality type defines a canonical stratification of a given algebroid or algebraic hypersurface over an algebraically closed field of characteristic zero.  
 Unfortunately, in general, no specific information on the polynomials of 
$\mathcal G$ is given in \cite{zariski1979}.  
Zariski's construction is purely transcendental, 
and  there is  no explicit bound on the degree of such polynomial 
projections. This makes, for instance, an algorithmic computation of Zariski's canonical stratification impossible.  The algebraic case was studied in more detail by Hironaka \cite{hironaka79}, where the semicontinuity of such a degree in Zariski topology is shown. This implies in particular that the dimensionality type induces, in the algebraic case, a stratification by locally closed algebraic subvarieties.  

For complex analytic singularities, we can define the dimensionality type using generic polynomial or generic analytic projections. It follows from Speder \cite[Theorems I-IV]{speder75}, that in the complex analytic case thus defined canonical stratification satisfies Whitney's conditions (a) and (b). It also follows from Theorems 3.7 and 7.3 of \cite{PP17}, where a stronger regularity condition, called (arc-w) was proven.  

It is an open question whether generic linear projections are generic in the sense of Zariski.    

\begin{question}
{6. Open problem}{Are generic linear projections sufficient  to define generic Zariski equisingularity and the dimensionality type ?  More precisely, is the formula  \eqref{eq:dimtype-polynproj} valid for a Zariski open non-empty set of linear projections $\pi_{\bar u} $ ?
} 
\end{question}

Even if the answers to the above questions were positive it would not give an algorithm to compute the dimensional type and the canonical stratification automatically, but the positive answer to this question would probably help to consider other related open problems that we summarize below.  


\begin{question}
{7. Open problems}{
Characterize geometrically or algebraically generic polynomial projections in the sense of  Zariski ? 
} 
\end{question}

In the case of strata of dimensionality type 2 a partial answer to 
both questions of problem 6 was obtained in \cite{brianconhenry80}.  
In this paper Brian\c con and Henry characterized 
generically Zariski equisingular families of isolated surface singularities
 in the 3-space in terms of local analytic invariants: Teissier's numbers (multiplicity, Milnor number, and Milnor number of a generic plane section), the number of double points and the number of cusps of the apparent contours of the generic projection of the generic  fibre of a mini-versal deformation.  All these numbers are local analytic invariants, and therefore linearly generic change of coordinates is generic in the sense of Zariski.  This shows in particular that, if $\Sing V$ is of codimension 2 in $V$ at $\point$, then the generic linear projections are sufficient to verify whether 
 $\dimtype (V,p) = 2$.  

Note that the answer to problem 7 should be quite subtle even in the case of the dimensionality type 2.  Let us remind that in 
\cite{luengo85} Luengo gave an example of a family of surface 
singularities in $\C^3$ that is Zariski equisingular for one transverse 
projection but not for the generic ones.

\begin{question}
{8. Open problems}{
Is the canonical Zariski's stratification of a complex analytic hypersurface Lipschitz equisingular ?  
} 
\end{question}

This problem is a version of problem 5. 
As we have mentioned in section 
\ref{ssec:lipschitz}, the positive answer for the strata of codimension 2  was announced \cite{PPpreprint19} and, if moreover $\codim (\Sing V) =2$, in \cite{NPpreprint}.

\subsection{Zariski equisingularity and equiresolution of singularities}\label{ssec:equiresolution}

In the case of dimensionality type $1$, i.e. equivalently, for families of plane curves singularities, Zariski equisingularity can be expressed  in terms of blowings-up (monoidal transformations) and equiresolution.  
More precisely, firstly, the following property of stability by blowings-up holds.

\begin{thm}[{\cite[Theorem 7.4]{zariski65-S2}}]\label{thm:equiresolution}
Assume that the singular locus $W$ of $V$ is of codimension $1$ and let $\point$ be a regular point of $W$. Let $\pi : V' \to V$ be the blowing-up of $W$, and let $\pointgen$ denote a general point of $W$.  
Then $V$ is of dimensionality type $1$ along $W$ at $\point$ if and only if the following two conditions hold 
\begin{enumerate}
\item  [\rm (1)]
$\pi^{-1} (\point) $ is a finite set of the same cardinality as 
$\pi^{-1} (\pointgen) $,   
\item [\rm (2)]
each $\point '\in \pi^{-1} (\point) $ is either a nonsingular point 
of $V'$ or a point of dimensionality type $1$. 
\end{enumerate}
\end{thm}

In the complex analytic set up  the conditions (1) and (2) mean that 
over a neighborhood of $\point$, $W':=\pi^{-1} (W) \to W$ is a finite analytic covering and $V'$ is Zariski generically equisingular along each connected component of $W'$ (this includes the case that $V'$ is nonsingular along this connected component). 

Secondly, the repeating process of such blowings-up leads not only to 
a resolution $\tilde \pi :\tilde V \to V$ of $V$ but also to an equiresolution in the following sense. Fix a local projection of $\pr : \K^{r+1}\to W$, such that $\pr ^{-1} (\point)$ is nonsingular and transverse to $W$. The fibers of this projection restricted to $W$ are plane curve singularities $V_t$ parameterized by $t\in W$.  Then the restrictions of  $\tilde \pi$, $\tilde V_t:= \tilde \pi^{-1}(V_t) \to V_t$ are the resolutions of $V_t$, see 
e.g. \mbox{
\cite[Corollary 7.5]{zariski65-S2} }
and the paragraph after it, and the induced  projections of $\tilde V$ and 
of the exceptional divisor $E$ of $\tilde \pi$  onto $W$ are submersions.   Note that $\tilde V$ coincides with a normalization of $V$, and hence it is also an equinormalization of the family $V_t$.  this can be deduced as well from Puiseux with parameter Theorem \ref{thm:PuiseuxTheorem}).

\begin{question}{Question}
Do the two properties, stability by blowing-up and equiresolution, hold 
for arbitrary codimension strata of Zariski's stratification ? 
\end{question}

The first part of this question was stated by Zariski in \mbox{
\cite{zariski76}}
: "Now, one test that any definition of equisingularity must meet is the test of its stable behavior along $W$ under blowing-up of $W$".  It also appears in questions F, G and H of \mbox{
\cite{zariski71open}}
.  An example of Luengo \mbox{
\cite{luengo84} }
shows that the generic Zariski equisingularity does not satisfy the stability under blowings-up property.     That is 
 $V= \{z^7+ y^ 7+ty^5x^3+ x^ {10}=0\}\subset \C^4$ is generically Zariski equisingular along $W=\Sing V= \{x=y=z=0\}$ but the blow-up $\tilde V$ of $V$ along $W$ is not generically Zariski equisingular along $\tilde W=\Sing \tilde V$  ($\tilde W $ is a nonsingular curve in this example).

 Moreover, reciprocally, a blowing-up may make non-equisingular families equisingular as shows another example from \mbox{
\cite{luengo84}}
.  In this example  
 $V= \{z^4+ y^6+tz^ 2 y^3+ x^ {8}=0\}\subset \C^4$ is not generically Zariski equisingular along $W=\Sing V= \{x=y=z=0\}$ (the origin is of dimensionality 3), but the strata of 
 Zariski canonical stratification of the blow-up of $W$, 
 $\tilde V \to V$,  project submersively onto $W$ (no point of the highest possible dimensionality $3$ in $\tilde V$).

In order to answer the second part of this question one has to make precise what the equiresolution, also often called simultaneous resolution, means. 
To start with there are embedded resolutions (modifications of the ambient space containing $V$) and the non-embedded ones. 
The concept of equiresolution was largely studied and clarified in \mbox{
\cite{Teissier76-77-80} }
within the context of non-embedded resolutions of complex analytic surface singularities, and in \mbox{
\cite{Lip97} }
in the context of embedded equiresolution of complex analytic or algebraic varieties.  

The relation between generic Zariski equisingularity and equiresolution depends which notion of the equiresolution is adapted.  In the first cited above example of Luengo 
\cite{luengo84}
, 
 $V$ understood as a  family $V_t$ does not have strong simultaneous resolution in the sense of Teissier 
\cite{Teissier76-77-80}
, if we require, moreover, that this resolution is given by a sequence blowings-up of non-singular equimultiple centers following Hironaka's algorithm, see 
\cite{luengo84} 
for details. 
 In \cite{Lip97} 
Lipman proposes a strategy to prove a weaker version of 
 equiresolution for such families.  This proof is completed by Villamayor in \
\cite{villamayor2000}
, in the algebraic case.   Villamayor's equiresolution is a modification of the ambient nonsingular space containing $V$ more general than the ones obtained as compositions of sequences of nonsingular centers blowings-up.  Moreover, it is not required that the induced resolution of $V$ is an isomorphism over $V \setminus \Sing V$.

There is another open problem related to the equiresolutions of 
families of singularities.  Namely, it is not clear whether, in general,  equiresolution can be used to construct topological trivializations. Let us explain it on a simplified example.  Suppose that $V$ is a hypersurface of a nonsingular (real or complex) analytic manifold $M$, $\tilde \pi : \tilde M \to M$ is a modification, the composition of blowings-up of smooth centers for instance, and that  
 \begin{enumerate}
 	\item [(i)]
 	there is a local (at $\point \in W$) analytic projection $\pr : M\to W$,
 	such that $\pr ^{-1} (\point)$ is nonsingular and transverse to $W$ and  whose fibers restricted to $V$ define a family of reduced hypersurfaces $V_t$, $t\in W$.
 	\item [(ii)]
 	$\tilde \pi ^{-1} (V)$ is a divisor with normal crossings that is the union of the strict transform $\tilde V$ of $V$, assumed non-singular, and the exceptional divisor $E$. 
 	\item[(iii)]
The strata of the canonical stratification of $\tilde \pi ^{-1} (V)$ (as a divisor with normal crossings) project by $\tilde \pr := \pr \circ \tilde \pi $ submersively onto $W$. 
 \end{enumerate}
 Then, by a version of Ehresmann fibration theorem, there is a trivialization of $\tilde \pr $ that preserves the strata of $\tilde \pi ^{-1} (V)$ and hence also $\tilde V$.  Moreover, by 
\cite{kuo85}, 
this trivialization can be made real analytic.  
 If this trivialization is a lift by $\tilde \pi$ of a trivialization of $\pr$ then we call the latter \emph{a blow-analytic trivialization}. 
 $$
\xymatrix{
\tilde M \ar@{{}->}[rr]^{\tilde \pi } \ar[rd]_{\tilde \pr } & & M  \ar[ld]^{\pr } \\
& W &
}
$$

In \mbox{
\cite{kuo85} }
Kuo developed the theory of blow-analytic equivalence of real analytic function germs.  Kuo shows that  
for families of isolated hypersurface (or function) singularities such 
blow-analytic trivializations exist under the following additional assumptions: $W = \Sing $ and $\tilde \pi $ is an isomorphism over the complement of $W$. In this case he constructs a nice (real analytic) trivialization of the resolution space that projects to a topological trivialization of $\pr$. In particular, it shows that the principle of generic blow-analytic equisingularity of real analytic function germs holds, see \cite[Theorem 1]{kuo85}.  But in the general, non-isolated singularity case it is not even clear whether there is a topological trivialization that lifts to the resolution space.  
The existence of a blow analytic trivialization of family of 
non-isolated singularities 
 remains the main open problem of Kuo's Theory, see \cite{FKK1998}, \cite{koike2004}.

Blow-analytic trivializations are, in particular, arc-analytic, and, actually, at least in the algebraic (i.e. Nash) case, blow-analytic and arc-analytic maps coincide, see \cite{bierstonemilmanarcanalytic}, \cite{subfunc}.  Thus there is a hope that Theorem \ref{thm:arcwise-analytic_triviality}, proven using Zariski equisingularity, can help in developing blow-analytic theory of non-isolated singularities.

\section{Appendix. Generalized discriminants} \label{sec:discriminants}

We recall the notion of generalized discriminants, 
see Appendix IV of \cite{whitneybook}, \cite{BPRbook}, \cite{roy2006}, or Appendix B of \cite{PP17}. 
Let $\kk$ be an arbitrary field of characteristic zero and let 
\begin{align}\label{type}
F(Z) = Z^d + \sum_{i=1}^ d a_iZ^{d-i} = \prod_{i=1}^d (Z-\xi_i)\in \kk [Z], 
\end{align}
be a polynomial with coefficients $a_i\in \kk$ and the roots $\xi_i\in \bkk$.  
Recall that the discriminant of $F$ is a polynomial in the coefficients 
$a_i$ that can be defined either in terms of the roots 
$$\D_F =  \prod_{1\le j_1< j_2 \le d}  (\xi_{j_2}-\xi_{j_1})^2 ,$$
or as the determinant of the Jacobi-Hermite matrix
$$\D_F =\left|\begin{array}{cccc} s_0 & s_1 & \cdots & s_{d-1}\\
s_1 & s_2 & \cdots & s_d \\
\cdots & \cdots & \cdots & \cdots\\
s_{d-1} & s_l & \cdots & s_{2d-2} \end{array}\right| ,$$
where  $s_i=\sum_{j=1}^d \xi_j^i$, for $i\in \N$, are Newton's symmetric functions.  Thus $\D_F =0$ if and only if $F$ has a multiple root in $\bkk$.  

\emph{The generalized discriminants, or subdiscriminants}, $\D_{d+1-l}$ of $F$, $l=1, \ldots, d$, 
can be defined as the principal minors of the  Jacobi-Hermite matrix 
$$\D_{d+1-l} : =\left|\begin{array}{cccc} s_0 & s_1 & \cdots & s_{l-1}\\
s_1 & s_2 & \cdots & s_l \\
\cdots & \cdots & \cdots & \cdots\\
s_{l-1} & s_l & \cdots & s_{2l-2} \end{array}\right| ,$$
and we put $D_d=1$ by convention. Thus $\D_{d+1-l}$ are polynomials in the coefficients $a_i$. 
The generalized discriminants can be defined equivalently in terms of the roots 
$$
\D_{d+1-l} =  \sum_{r_1 <\cdots < r_{l}} \,  \prod_{j_1< j_2;\, j_1, j_2 \in \{r_1, \ldots ,r_{j}\}} (\xi_{j_2}-\xi_{j_1})^2. 
$$

In particular $\D_1= \D_F$ and $F$ admits exactly $l$ distinct roots in $\overline \kk$ if and only if $\D_{1} = \cdots = \D_{d-l}=0$ and 
$\D_{d-l+1} \ne 0$.

If $F$ is not reduced, that means in this case that has multiple roots, the generalized discriminants can replace the (classical) 
discriminant of $F_{red}$. Here $F_{red}$ equals $\prod (Z-\xi_i)$, 
where the product is taken over all distinct roots of $F$. 
 The following lemma is easy. 

\begin{lemma}\label{lem:twodiscr} 
Suppose $F$ has exactly $l>1$ distinct roots in 
$\bkk$ of multiplicities $\mathbf m=(m_1, ..., m_l)$.  Then there is a positive constant $C= C_{l,\mathbf m}$, depending only on 
$\mathbf m=(m_1, ..., m_l)$, such that  the generalized discriminant 
$\D_{d-l+1}$ of $F$ and 
the standard discriminant $\D_{F_{red}}$ of $F_{red}$ are related by 
$$
\D_{d-l+1}= C \D_{F_{red}}.
$$
\end{lemma}

We conclude with the following obvious consequence of the IFT.

\begin{lemma}\label{lem:implicit}
Let $F\in \C \{x_1, ... ,x_n\}[Z]$ be a monic polynomial in $Z$ such that the discriminant $\D_{F_{red}}$ does not vanish at the origin.  Then, there is a neighborhood $U$ of $0\in \C^n$ such that 
the complex roots $F$ are analytic on $U$, distinct, and of constant multiplicities (as the roots of $F$).  
\end{lemma}



\bibliographystyle{siam}
\bibliography{ZE}
\printindex

\end{document}